\documentclass[11pt,leqno,a4paper, english]{smfart}

\usepackage[]{hyperref}
\hypersetup{
    colorlinks=true,       % false: boxed links; true: colored links
    linkcolor=red,          % color of internal links
    citecolor=blue,        % color of links to bibliography
    filecolor=magenta,      % color of file links
    urlcolor=cyan           % color of external links
}

\usepackage{amsmath}
\usepackage{amsfonts,amssymb}
\usepackage{enumerate}
\usepackage{mathrsfs}
\usepackage{amsthm}
\usepackage{a4wide}

\theoremstyle{plain}
\newtheorem{thm}{Theorem }[section]
\newtheorem{prop}[thm]{Proposition}
\newtheorem{lemm}[thm]{Lemma}
\newtheorem{coro}[thm]{Corollary}

\theoremstyle{definition}
\newtheorem{rema}[thm]{Remark}

\DeclareMathOperator{\supp}{supp}

\DeclareSymbolFont{pletters}{OT1}{cmr}{m}{sl}
\DeclareMathSymbol{s}{\mathalpha}{pletters}{`s}

\def\eps{\varepsilon}

\def\la{\left\lvert}
\def\lA{\left\lVert}

\def\L#1{\langle #1 \rangle}

\def\mez{\frac{1}{2}}

\def\ra{\right\rvert}
\def\rA{\right\rVert}

\def\xC{\mathbf{C}}
\def\xN{\mathbf{N}}
\def\xR{\mathbf{R}}
\def\xS{\mathbf{S}}
\def\xT{\mathbf{T}}
\def\xZ{\mathbf{Z}}

\numberwithin{equation}{section}

\title[concentration of eigenfunctions]{Concentration of Laplace eigenfunctions  and stabilization of weakly damped wave equation}
\author{N.Burq}
\address{Laboratoire de Math\'ematiques UMR 8628 du CNRS. Universit\'e Paris-Sud, B\^atiment 425, 91405 Orsay Cedex}
 \email{nicolas.burq@math.u-psud.fr}
 \author{C.Zuily }
 \address{Laboratoire de Math\'ematiques UMR 8628 du CNRS. Universit\'e Paris-Sud, B\^atiment 425, 91405 Orsay Cedex}
 \email{claude.zuily@math.u-psud.fr}
 \thanks{N.B. was supported in part by Agence Nationale de la Recherche
  project NOSEVOL, 2011 BS01019 01. N.B. and C. Z. were supported in part by Agence Nationale de la Recherche
  project  ANA\'E ANR-13-BS01-0010-03.}
\date{\empty}
\begin{abstract}
In this article, we prove some universal bounds on the speed of concentration on small (frequency-dependent) neighborhoods of submanifolds of $L^2$-norms of quasi modes for Laplace operators on compact manifolds. We deduce new results on the rate of decay of weakly damped wave equations.  
\end{abstract}
\begin{altabstract}
On d\'emontre dans cet article des bornes universelles sur la vitesse de concentration dans de petits voisinages (d\'ependant de la fr\'equence) de sous vari\'et\'es pour les normes $L^2$ de quasimodes du Laplacian sur une vari\'et\'e compacte. On en d\'eduit de nouveaux r\'esultats sur la d\'croissance des \'equations des ondes faiblement amorties.
\end{altabstract} 

\begin{document}
 \maketitle
 \section{Notations and main results}
   Let  $(M,g)$ be a smooth    compact  Riemanian manifold without boundary of  dimension $n$, $ \Delta_g$ the Laplace-Beltrami operator on $M$ and  $d(\cdot,\cdot)$  the geodesic   distance on  $M$. 
   
   The purpose of this work is to investigate the concentration properties of eigenfunctions of the  operator $ \Delta_g$ (or more generally quasimodes). There are many ways of measuring such possible concentrations. The most classical is by describing semi-classical (Wigner) measures (see the works by Shnirelman~\cite{Sh74}, Zelditch~\cite{Ze87}, Colin de Verdi\`ere~\cite{CdV85}, G\'erard-Leichtnam~\cite{GeLe93-1}, Zelditch-Zworski~\cite{ZeZw96}, Helffer-Martinez-Robert~\cite{HeMaRo87}. Another approach was iniciated by Sogge and consists in the studying the potential growth of $\|\varphi_\lambda\|_{L^p(M)}$, see the works by Sogge~\cite{So88, So93}, Sogge-Zelditch~\cite{SoZe01}, Burq-G\'erard-Tzvetkov~\cite{BuGeTz03-1, BuGeTz03-2, BuGeTz04}. Finally in~\cite{BuGeTz05, BoRu12,Ta10} the concentration of restrictions on submanifolds was considered. Here, we focus on a situation intermediate between the latter (concentration on submanifolds) and the standard $L^2$-concentration (Wigner measures). Indeed, we study the concentration (in $L^2$ norms) on small (frequency dependent) neighborhoods of submanifolds. 
   Our first result is the following 
      \begin{thm}\label{quasi-mode}
    Let $k \in \{1,\ldots, n-1\}$ and $\Sigma^k$ be a submanifold of dimension $k$ of   $M$. Let 
    Let us introduce for $\beta >0$,
\begin{equation}\label{N}
N_{\beta} = \{ p\in M: d(p, \Sigma^k) < \beta\}.
\end{equation}
There exists $C>0, h_0>0$ such that for every $0<h\leq h_0$, every $\alpha \in (0,1)$ and every  solution $\psi  \in H^2(M)$ of the equation on $M$
$$(h^2 \Delta_g   +1)\psi  = g $$
we have the estimate
\begin{equation}\label{qm-N}
\Vert \psi  \Vert_{L^2(N_{\smash{\alpha h^{1/2}}})} \leq C\alpha^\sigma\big( \Vert \psi  \Vert_{L^2(M)} + \frac{1}{h}\Vert g  \Vert_{L^2(M)}\big)
 \end{equation}
 where  $\sigma = 1$ if $k \leq n-3$, $\sigma = 1^-  $ if $k = n-2, \sigma = \mez $ if $k = n-1.$   
 \end{thm}
 Here $1^-$ means that we have a logarithm loss i.e. a bound by $C \alpha \vert \log (\alpha) \vert.$
 \begin{rema}
 As pointed to us by M. Zworski, the result above is {\em not} invariant by conjugation by Fourier integral operators. Indeed, it is well known that micro locally, $-h^2 \Delta -1$ is conjugated by a (micro locally unitary) FIO to the operator $hD_{x_1}$. However the result above is clearly false is one replaces the operator $-h^2 \Delta -1$ by $hD_{x_1}$
 \end{rema}
 Another motivation for our study was the question of stabilization for weakly damped wave equations.  \begin{equation}\label{damped}
   (\partial_t^2 - \Delta_g+ b(x) \partial_t )u =0, (u, \partial_tu ) \mid_{t=0} ( u_0, u_1)\in H^{s+1}(M) \times H^s(M), 
   \end{equation}
where $0\leq b \in L^\infty (M)$. Let 
$$ E(u)( t) = \int_{M} \big\{g_p(\nabla_g u(p), \nabla_g u(p)) + | \partial_t u(p)|^2 \big\}dv_g(p)
$$
where $\nabla_g$ denotes the gradient with respect to the metric $g$.

It is known that as soon as the damping $b\geq 0$ is non trivial, the energy of  every solution converge to $0$ as $t$ tends to infinity.  On the other hand the rate of decay is {\em uniform} (and hence exponential) in energy space if and only if the {\em geometric control condition}~\cite{BaLeRa92, BuGe96} is satisfied. Here we want to explore the question when some trajectories are trapped and exhibit decay rates (assuming more regularity on the initial data).  This latter question was previously studied in a general setting in~\cite{Le96} and on tori in ~\cite{BuHi05, Ph07, AL14} (see also~\cite{BuZw03, BuZw03-1}) and more recently by Leautaud-Lerner~\cite{LeLe14}.   
According to the works by Borichev-Tomilov~\cite{BoTo10}, stabilization results for the wave  equation are equivalent to resolvent estimates. On the other hand,  Theorem~\ref{quasi-mode} implies easily (see Section~\ref{se.2.2}) the following resolvent estimate
\begin{coro}\label{corsimple}
Consider for $h>0$ the following operator
 \begin{equation}\label{L}
 L_h= - h^2 \Delta_g -1 + i h b, \quad   b \in L^\infty(M). 
\end{equation}
   Assume that there exists a {\em global} compact submanifold $\Sigma^k\subset M$ of dimension $k$ such that 
\begin{equation}\label{hyp-b}
  b(p) \geq C d( p, \Sigma^k)^{2\kappa}, \quad p \in M 
  \end{equation} 
for some $\kappa  >0.$ Then there exist  $C>0, h_0>0$ such that for all $ 0<h\leq h_0$ 
$$\Vert \varphi  \Vert_{L^2(M)} \leq C h^{-(1+ \kappa)} \Vert L_h \varphi \Vert_{L^2(M)},$$
  for all  $\varphi  \in H^2(M)$.
  \end{coro}
  This result  will  imply the following one. 
\begin{thm}\label{example}
Under the geometric assumptions of Corollary~\ref{corsimple},  there exists $C>0$ such that for any $(u_0, u_1) \in   H^{2}(M) \times H^1(M)$, the solution $u$ of~\eqref{damped} satisfies 
\begin{equation}
\label{taux}
 E (u (t))^\mez \leq \frac{ C} {t^{\frac 1 \kappa}} \bigl( \| u_0\|_{H^{2}}+ \|u_1\|_{H^1}\bigr).
 \end{equation}
\end{thm}
\begin{rema}
Notice that in Theorem~\ref{example} the decay rate is worst than the rates exhibited by Leautaud-Lerner~\cite{LeLe14} in the particular case when the submanifold $\Sigma$ is a torus (and the metric of $M$ is flat near $\Sigma$). We shall exhibit  below examples showing  that the rate~\eqref{taux} is optimal in general.
\end{rema}
A main drawback of the result above (and Leautaud-Lerner's results) is that we were led to {\em global} assumptions on the geometry of the manifold $M$ and the trapped region $\Sigma^k$. However, the flexibility of Theorem~\ref{quasi-mode} is such that we can actually dropp all {\em global} assumptions and keep only a {\em local} weak controlability assumption.
 \begin{thm}\label{stabi}
Let us assume the following {\em weak geometric control property}: 
for any $\rho_0 = ( p_0, \xi_0) \in S^* M$, there exists $s\in \xR$ such that the point $(p_1, \xi_1)= \Phi(s) (\rho_0)$ on the bicharacteristic issued from $\rho_0$ satisfies 
\begin{itemize}
\item either $p_1 \in \omega= \cup \{ U \text{ open }; \text{essinf} _U \, b  >0 \}$ 
\item or there exists $\kappa>0, C>0$ and a local submanifold $\Sigma^k$ of dimension $k\geq 1$ such that $p_1\in \Sigma^k$ and   near $p_1$,  
$$b(p) \geq Cd(p, \Sigma^k)^{2\kappa}. $$ 
\end{itemize}

Notice that since $S^*M$ is compact,  we can assume in the assumption above that $s\in [-T, T]$ is bounded and that a finite number of submanifolds (and kappa's) are sufficient. Let $\kappa_0$ be the largest. Then   there exists $C >0$ such that for any $(u_0, u_1) \in   H^{2}(M) \times H^1(M)$, the solution $u$ of~\eqref{damped} satisfies 
$$ E (u (t))^\mez \leq \frac{ C } {t^{\frac {1} {\kappa_0}}} \bigl( \| u_0\|_{H^{2}}+ \|u_1\|_{H^1}\bigr).$$
\end{thm}

The results in Theorem~\ref{quasi-mode} are in general optimal. On spheres $\xS^n= \{ x\in \xR^{n+1} :  |x| =1\}$, an explicit family of eigenfunctions $ e_j(x_1, \dots, x_{{n+1}})= (x_1+ i x_2)^j$ (eigenvalues $\lambda_j^2 = j(j+n-1)$) is known. We have
\begin{equation}
\label{expo}
 | e_j (x)| ^2 = ( 1- |x'|^2)^j = e^{j \log( 1-|x'|^2)}, \quad x' = (x_3, \dots, x_{n+1}),
 \end{equation} and consequently, these eigenfunctions concentrate exponentially on $j^{-1/2}$ neighborhoods of the geodesic curve given by $\{ x \in \xS^n; x'=0\}$ (the equator). 
As a consequence, the sequence  $j^{\frac {d-1} 4} e_j$ is (asymptotically) normalized by a constant in $L^2( \xS^n)$,  and if $\Sigma^k$ contains the equator, we can see optimality. Indeed, we  work in local coordinates $(y, x')$ where $y \in \xT$ and $x'\in V$ close to $0$ in $ \xR^{n-1}$. This localization being licit since according to~\eqref{expo}, the fonction is $\mathcal{O}( e^{-\delta j})$ outside of a fixed neighborhood of the equator.  Let $h= j^{-1}$,Let us decompose $x'= (y', z') \in \xR^{k-1} \times \xR^{n-k}$ and consider the submanifold defined by $z'=0$.Then 
$$ \| e_j\|_{L^2(N_{\alpha h^{1/2}})}\sim \int_{y=0}^{1} \int_{|y'| \leq 1 } \int_{|z'| \leq \alpha h^{1/2}} j^{\frac {n-1} 2} e^{- j ( |y'|^2 + |z'|^2)} dy' dz' \sim \alpha ^{n-k}.$$
This elementary calculation shows that our results are saturated for all $\alpha >0$ on spheres (including the exponent of $\alpha$ appearing in~\eqref{qm-N}) by eigenfunctions in the case submanifolds of codimension $1$ or $2$ (except for the logarithmic loss appearing in the case of codimension $2$). On the other hand again on spheres, other particular families of eigenfunctions, $(f_j, \lambda_j)$ are known (the so called zonal spherical harmonics). These are known to have size of order \smash{$\lambda_j^{(n-1)/ 2}$} in a neighborhood of size $\lambda_j ^{-1}$ of two antipodal points (north and south poles). 
As a consequence, a simple calculation shows that if the submanifold contains such a point (which if always achievable by rotation invariance), we have, for $\alpha = \epsilon h^{1/2}$
$$\| f_j\|^2_{L^2(N_{\alpha h^{1/2}})} \geq c h\sim  \alpha^2,$$
which shows that ~\eqref{qm-N} is optimal on spheres (at least in the regime $\alpha \sim h^{1/2}$).
To get the full optimality might be possible by studying other families of spherical harmonics. For general manifolds, following the analysis in~\cite[Section 5]{BuGeTz05}) should give the optimality of our results for {\em quasi-modes} on any manifold.

The paper is organized as follows. We first show how the non concentration result (Theorem~\ref{quasi-mode}) imply resolvent estimates for the damped Helmholtz equation, which in turn imply very classically the stabilization results for the damped wave equation. We then focus on the core of the article and prove Theorem~\ref{quasi-mode}. We start with the case of curves for which  we have an alternative proof. Then we focus on the general case. We first show that the resolvent estimate is implied by a similar estimate for the spectral projector. To prove this latter estimate, we rely harmonic analysis and the precise description of the spectral projector given in~\cite{BuGeTz05}. Finally, we gathered in an appendix several  technical results. 

{\bf Acknowledgments} We'd like to thank M. Zworski for fruitful discussions about these results.

\section{From concentration estimates to stabilization results}
\subsection{{\em A priori} estimates}
Recall that $(M,g)$ is a compact connected Riemanian manifold. We shall denote by $\nabla_g$ the gradient operator with respect to the metric $g$ and by $dv_g$ the canonical volume form on $M.$   In all this section we set
\begin{equation}\label{op-L}
L_h = -h^2 \Delta_g -1 +ihb
\end{equation}
We shall first derive some a-priori estimates on $L_h.$ 
 \begin{lemm}\label{energie}
 Let $ L_h= - h^2 \Delta_g -1 + i h b.$
 Assume $b \geq 0 $ and set $f  = L_h \varphi $. Then
 \begin{equation}
 \begin{aligned}
 &(i) \quad h \int_M b \vert \varphi (p) \vert^2\, dv_g(p) \leq \Vert \varphi  \Vert_{L^2(M)}\Vert f  \Vert_{L^2(M)},\\
 &(ii) \quad h^2\int_Mg_p \big (\nabla_g \varphi (p), \overline{\nabla_g \varphi (p)} \big) \, dv_g(p) \leq  \Vert \varphi  \Vert^2_{L^2(M)} +\Vert \varphi  \Vert_{L^2(M)}\Vert f  \Vert_{L^2(M)}.
\end{aligned}
\end{equation}
   \end{lemm}
\begin{proof}
We know that $\Delta_g = \text{div} \nabla_g $ and by the definition of these objects we have
$$ A=: \int_Mg_p \big (\nabla_g\varphi (p), \overline{\nabla_g \varphi (p)} \big) \, dv_g(p) = - \int_M \Delta_g \varphi (p) \overline{\varphi (p)}\, dv_g(p).$$
Multipying both sides by $h^2$ and since $- h^2 \Delta_g\varphi = f + \varphi  - i h b \varphi$  we obtain
$$h^2A =  \int_M  \vert  \varphi (p)\vert^2 \, dv_g(p) - ih  \int_M  b(p) \vert  \varphi (p)\vert^2 \, dv_g(p) + \int_M f (p)\overline{\varphi (p)} \, dv_g(p).$$
Taking the real  and the imaginary parts of this equality we obtain   the desired estimates.
  \end{proof}
  \subsection{Proof of Corollary \ref{corsimple} assuming Theorem \ref{quasi-mode}} \label{se.2.2}
 According to condition \eqref{hyp-b} we have  on $ N_{\smash{\alpha h^{1/2}}}^c$ 
$$b(p) \geq C  d(p, \Sigma^k)^{2\kappa} \geq C  \alpha^{2 \kappa} h^\kappa.$$
Writing $\int_{ N_{\smash{\alpha h^{1/2}}}^c} \vert \varphi (p)\vert^2\, dv_g(p) = \int_{ N_{\smash{\alpha h^{1/2}}}^c} \frac{1}{b(p)} b(p)\vert \varphi (p)\vert^2\, dv_g(p),$
 we deduce from Lemma \ref{energie} that
\begin{equation}\label{est-Nc}
\int_{ N_{\smash{\alpha h^{1/2}}}^c} \vert \varphi (p)\vert^2\, dv_g(p)   \leq  \frac{1}{C  \alpha^{2 \kappa}}h^{-(1+ \kappa)} \Vert \varphi \Vert_{L^2(M)}\Vert f \Vert_{L^2(M)}.
\end{equation}
   Therefore we are left with the estimate of the $L^2(N_{\alpha h^{1/2}})$ norm of $\varphi.$  
 
According to \eqref{op-L} we see that $\varphi $ is a solution of 
$$(h^2 \Delta_g   +1)\varphi  = -f  + i hb\varphi  =:g_h$$
where $g_h$ satisfies
$$\Vert g_h \Vert_{L^2(M)} \leq \Vert f  \Vert_{L^2(M)} + Ch \Vert \varphi \Vert_{L^2(M)}.$$
It follows from \eqref{est-Nc} and Theorem \ref{quasi-mode} that 
$$
  \Vert \varphi  \Vert_{L^2(M)} \leq   \frac{1}{C^\mez  \alpha^\kappa} h^{-\frac{1+ \kappa}{2}} \Vert \varphi  \Vert^\mez_{L^2(M)}\Vert f  \Vert^\mez_{L^2(M)} 
+    C \alpha^\sigma  (\Vert \varphi  \Vert_{L^2(M)}  + \frac{1}{h} \Vert f  \Vert_{L^2(M)}).
$$
Now we fix   $\alpha$ so small that $C  \alpha^\sigma \leq \mez$ and we use the inequality $a^\mez b^\mez \leq \eps a + \frac{1}{4 \eps} b$ to obtain eventually
$$\Vert \varphi   \Vert_{L^2(M)} \leq C'h^{-(1+ \kappa)}\Vert f \Vert_{L^2(M)}$$
which completes the proof of Corollary \ref{corsimple}.
\subsection{Proof of Theorem \ref{example} assuming Corollary \ref{corsimple}}
 The proof  is an immediate consequence of   a work by Borichev-Tomilov~\cite{BoTo10} and  Corollary \ref{corsimple}. We quote the following proposition from  \cite[Proposition 1.5]{LeLe14}.
 \begin{prop}
 Let $\kappa >0.$ Then the estimate \eqref{taux} holds if and only if there exist positive constants $C, \lambda_0$ such that for all $u \in H^2(M)$, for all $ \lambda \geq \lambda_0$ we have
 $$C \Vert (- \Delta_g  -\lambda^2 + i\lambda b )u \Vert_{L^2(M)} \geq \lambda^{1- \kappa} \Vert u \Vert_{L^2(M)}.$$
\end{prop}
 
 \subsection{Proof of Theorem \ref{stabi} assuming Theorem \ref{quasi-mode}}
 As before Theorem \ref{stabi} will follow from the resolvent estimate 
 \begin{equation}\label{res-est}
    \exists  C>0, h_0>0:   \forall h \leq h_0 \quad \Vert \varphi \Vert_{L^2(M)} \leq C h^{-(1+ \kappa)} \Vert L_h \varphi \Vert_{L^2(M)}
\end{equation}
for   every $\varphi \in C^\infty(M).$

We   prove \eqref{res-est} by contradiction. If it is false   one can find sequences $(\varphi_j),  (h_j), (f_j)$ such that
 \begin{equation}\label{non-res-est}
    (-h_j^2 \Delta_g -1 +ih_jb) \varphi_n = f_j  \quad \text{and} \quad \Vert \varphi_j \Vert_{L^2(M)} > \frac{j}{h_j^{1+ \kappa}} \Vert f_j \Vert_{L^2(M)}.
  \end{equation}
Then $\Vert \varphi_j \Vert_{L^2(M)} >0$ and we may therefore assume that $\Vert \varphi_j\Vert_{L^2(M)} =1.$ It follows that 
\begin{equation}\label{fn=o}
  \Vert f_j \Vert_{L^2(M)} = o(h_j^{1+\kappa}), \quad j \to + \infty.
  \end{equation}
 Let $\mu$ be a semiclassical measure for $(\varphi_j).$  By Lemma \ref{energie} we have
 $$  \la \int_M \big\{\vert h_j \nabla_g \varphi_j(p) \vert^2 - \vert \varphi_j(p) \vert^2 \big \} \, dv_g(p) \ra \leq \Vert f_j \Vert_{L^2(M)}.$$
   It follows that  $(\varphi_j)$ is $h_j$-oscillating which implies that  $\mu (S^*(M))=1.$ We therefore shall reach a contradiction if we can show that $\supp \mu = \emptyset$ and \eqref{res-est} will be proved. 
   First of all by elliptic regularity we have
 \begin{equation}\label{ell-reg}
 \supp \mu \subset \{(p, \xi) \in S^*(M) : g_p(\xi, \xi) =1\}.
\end{equation}
On the other hand using Lemma \ref{energie} we have
 \begin{equation}\label{n-r-e-2}
    \int b(p) \vert \varphi_j(p) \vert^2\, dv_g(p) \leq \frac{1}{h_j}Ê\Vert f_j \Vert_{L^2(M)}  
\end{equation}
since $\Vert \varphi_j \Vert_{L^2(M)} =1.$  We deduce from \eqref{non-res-est}, \eqref{n-r-e-2}  and \eqref{fn=o} that 
\begin{equation}\label{est-gn}
  (h_j^2 \Delta_g +1)\varphi_j = G_j, \quad \text{where }\quad  \Vert G_j \Vert_{L^2(M)} = o(h^{1+\frac{\kappa}{2}}_j) \to 0, j \to + \infty.
  \end{equation}
  This shows that the support of $\mu$ is invariant by the geodesic flow. Let $\rho_0 \in S^*(M)$ and $\rho_1= (p_1, \xi_1) \in S^*(M))$ belonging to the geodesic issued from $\rho_0$. Then
  $$\rho_0 \notin \supp \mu \Longleftrightarrow \rho_1 \notin \supp \mu.$$
  But according to our assumption of weak geometric control,     either a neighborhood of $p_1$ belongs to the set $\{b(p) \geq c>0\}$ or $p_1 \in \Sigma^k$ and   $b(p) \geq Cd(p, \Sigma^k)^{2 \kappa}$  near $p_1.$ In the first case  in a neighborhood of $\rho_1$ the essential inf of $b$ is positive and hence by \eqref{n-r-e-2} $\rho_1 \notin \supp \mu.$ In the second case taking a small neighborhood $\omega$ of $p_1$ we write
$$  \int_\omega \vert \varphi_j(p)\vert^2 \, dv_g(p) = \int_{\omega \cap N_{\smash{\alpha h_j^{1/2}}}} \vert \varphi_j(p)\vert^2 \, dv_g(p) + \int_{\omega \cap N^c_{\smash{\alpha h_j^{1/2}}}} \vert \varphi_j(p)\vert^2 \, dv_g(p) = (1) +(2).$$
By Theorem \ref{quasi-mode} and \eqref{est-gn} we have
$$ (1) \leq C \alpha^\kappa(1+ \frac{1}{h_j}\Vert g_j \Vert_{L^2(M)}) \leq C \alpha^\kappa(1+o(h_j^\frac{\kappa}{2})) $$
  Using the assumption $b(p) \geq C d(p, \Sigma^k)$ and \eqref{n-r-e-2}    we get
$$  (2)   \leq \frac{C}{\alpha^{2\kappa} h_j^\sigma} \int_M b(p) \vert \varphi_j(p)\vert^2 \, dv_g(p) \leq \frac{C'Ê\Vert f_j \Vert_{L^2(M)}}{\alpha^{2\kappa} h_j^{1+\kappa}} = \frac{o(1)}{\alpha^{2\kappa}}.$$
 It follows that 
 $$  \int_\omega \vert \varphi_j(p)\vert^2 \, dv_g(p) \leq C \alpha^\kappa  +   \frac{o(1)}{\alpha^{2\kappa}}.$$
Let $\eps >0$. We first fix $\alpha(\eps) >0$ such that $ C \alpha(\eps)^\sigma \leq \mez \eps$ then we take $j_0$ large enough such that for $j \geq j_0, o(1) \leq  \alpha(\eps)^{2\kappa}  \mez \eps.$ It follows that for $j \geq j_0$ we have $\int_\omega \vert \varphi_j\vert^2 \, dv_g \leq \eps.$ This shows that $\lim_{j \to + \infty}  \int_\omega \vert \varphi_j\vert^2 \, dv_g = 0$ which implies that $\rho_1 \notin \supp \mu $ thus $\rho_0 \notin \supp \mu. $ Since $\rho_0$ is arbitrary we deduce that $\supp \mu = \emptyset$ which the desired contradiction.
\section{Concentration estimates (Proof of Theorem \ref{quasi-mode})} 
The rest of the paper will be devoted to the proof of Theorem \ref{quasi-mode}. The case $k=1$ i.e. the case of curves, is easier, so we shall start by   this case before dealing with  the general case.
   \subsection{The case of curves} 
   In this case we follow the strategy in~\cite[Section 2.4]{BuGeTz04-3},~\cite{KoTaZw07} and see the equation satisfied by quasi modes as an evolution equation with respect to a well chosen variable. 
      One can find an open neighborhood $U_p$ of $p$ in $M$, a neighborhood $B_0$ of the origin in $\xR^n$ a diffeomorphism $\theta$ from $U_p $ to $B_0$ such that
\begin{equation*}
\begin{aligned}
&(i)\quad \theta (U_p \cap \Sigma^1) = \{x =(x',x_n) \in (\xR^{n-1}\times \xR) \cap B_0: x'=0\}\\
&(ii) \quad \theta(N) \subset \{x\in B_0: \vert x' \vert \leq \alpha h^\mez\}.
\end{aligned}
 \end{equation*}
 Now $\Sigma^1 $ is covered by a finite number of such open neighborhoods i.e. $\Sigma^1 \subset \cup_{j=1}^{n_0} U_{p_j}.$ We take a partition of unity relative to this covering i.e. $(\chi_j) \in C^\infty(M)$ with $\supp \chi_j \in U_{p_j}$ and $\sum_{j=1}^{n_0} \chi_j = 1$ in a fixed neighborhood of $\Sigma^1.$ 
 Taking $h$ small enough  we can write
 $$ \psi_h = \sum_{j=1}^{n_0} \chi_j \psi_h, \quad  (  h^2 \Delta_g +1) \psi_h =   \sum_{j=1}^{n_0}  (h^2 \Delta_g +1)(\chi_j \psi_h)\quad \text{on } N.$$
Now   for $j = 1, \ldots, n_0$ set
\begin{equation}\label{Fjh}
F_{j,h}= (h^2 \Delta_g +1)(\chi_j \psi_h).
\end{equation}
Then
   $$F_{j,h}= \chi_j g_h -  h^2(\Delta_g \chi_j) \psi_h -2h^2g_p(\nabla_g\psi_h, \nabla_g \chi_j )    =: (1) -(2)-(3). $$
  We have $\Vert(1)\Vert_{L^2(M)} \leq C\Vert g_h \Vert_{L^2(M)} $ and $\Vert(2)\Vert_{L^2(M)} \leq Ch^2\Vert \psi_h\Vert_{L^2(M)}.$ By the Cauchy Schwarz inequality we can write 
$$ h^2g_p(\nabla_g \psi_h, \nabla_g \chi_j ) \leq h^2g_p(\nabla_g \psi_h, \nabla_g \psi_h )^\mez g_p( \nabla_g\chi_j, \nabla_g \chi_j )^\mez$$ 
which implies that  $\vert(3)\vert^2 \leq Ch^4 g_p(\nabla_g \psi_h, \nabla_g \psi_h ).$ 
It follows from Lemma \ref{energie} with $b \equiv 0$ that $\Vert(3)\Vert_{L^2(M)} \leq Ch (\Vert \psi_h\Vert_{L^2(M)} +  \Vert \psi_h \Vert^\mez_{L^2(M)}\Vert g_h \Vert^\mez_{L^2(M)}).$  

Summing up we have proved that for $ j = 1, \ldots, n_0$
\begin{equation}\label{tronc}
  \Vert F_{j,h} \Vert_{L^2(M)} \leq C(h  \Vert \psi_h \Vert_{L^2(M)}   +  \Vert g_h \Vert_{L^2(M)}).
  \end{equation}
Setting $u_{j,h}(x) = (\chi_j \psi_h) \circ \theta_j^{-1}(x) $ we see that we have
\begin{equation}\label{somme}
\Vert \psi_h \Vert_{L^2( N_{\smash{\alpha h^{1/2}}})} \leq \sum_{j=1}^{n_0} \Vert \chi_j \psi_h\Vert_{L^2( N_{\smash{\alpha h^{1/2}}})} \leq C \sum_{j=1}^{n_0} \Vert u_{j,h}\Vert_{L^2(B_{\alpha,h})}
\end{equation}
where $B_{\alpha,h} =\{x=(x',x_n)\in \xR^{n-1}\times \xR: \vert x' \vert \leq \alpha h^\mez, \vert x_n \vert \leq c_0\}$.

 Our aim is to bound $ \Vert u_{j,h}\Vert_{L^2(B_{\alpha,h})}$ for $j = 1, \ldots, n_0.$ Therefore we can fix $j$ and omit it in what follows.  Without loss of generality we can assume that $\supp u_h \subset K$ where $K$ is a fixed compact independent of $h$.
 
 Notice that   Lemma \ref{energie} with $b \equiv 0$ implies that
   \begin{equation}\label{grad}
    \begin{aligned}
 h^2 \Vert \nabla_x u_h \Vert^2_{L^2(\xR^n)} \leq C (\Vert \psi_h \Vert^2_{L^2(M)} +\Vert \psi_h\Vert_{L^2(M)}\Vert f_h \Vert_{L^2(M)})
 \end{aligned}
  \end{equation}
    From \eqref{tronc} we see that
\begin{equation}
    ( h^2 P +1)u_{h} = G_{h}, \text{where}\quad 
           \end{equation}
   where $P = \frac{1}{g(x)^\mez}\sum_{k,l =1}^n \frac{\partial}{\partial x_k}\big( g(x)^\mez g^{kl}(x)\frac{\partial}{\partial x_l}\big)$ is the image of the Laplace Beltrami operator under the diffeomorphism  and 
   \begin{equation}\label{Fh}
  \Vert G_{h}\Vert_{L^2(\xR^n)} \leq  C(h  \Vert \psi_h \Vert_{L^2(M)}   +  \Vert g_h \Vert_{L^2(M)}).
\end{equation}
Now let $\Psi_1 \in C^\infty(\xR^n), \Psi_1(\xi) = \mez$ if $ \vert \xi \vert \leq 1, \Psi_1(\xi) = 0 $ if $\vert \xi \vert \geq 1 $ and $\Psi \in C_0^\infty(\xR^n), \Psi=1$ on the support of $\Psi_1.$ Then $(1-\Psi_1)(1-\Psi) = 1-\psi.$
 Write
\begin{equation}\label{uh}
u_h = (I- \Psi(hD)) u_h + \Psi(hD)u_h =: v_h + w_h.
\end{equation}
We have
$$   (h^2 P +1)v_{h}=(h^2 P +1) (I-\Psi_1(hD))v_{h} =   (I- \Psi_1(hD)) \widetilde{F}_{h} -\big[h^2P, \Psi_1(hD)\big]u_h=: G_{1h}$$
By   \eqref{Fh}, \eqref{grad} and the semi classical symbolic calculus we have
$$ \Vert G_{1h}\Vert_{L^2(\xR^n)} \leq  C(h  \Vert \psi_h \Vert_{L^2(M)}   +  \Vert g_h \Vert_{L^2(M)}).$$
Now on the support of $ 1- \Psi_1(\xi) $, the principal symbol of the semi classical pdo,  $Q= (h^2 P +1)$ does not vanish. By the elliptic regularity we have therefore
\begin{equation}\label{vh}
\sum_{k =0}^2\Vert (h \nabla)^kv_h \Vert_{L^2(\xR^n)} \leq C \Vert G_{1h} \Vert_{L^2(\xR^n)}\leq C(h  \Vert\psi_h \Vert_{L^2(M)}   +  \Vert g_h \Vert_{L^2(M)}).
\end{equation}
It follows that for $\eps>0$ small we have
\begin{equation}\label{vh1}
h^{1+ \eps} \Vert  v_h \Vert_{H^{1+\eps}(\xR^n)} \leq  C(h \Vert \psi_h \Vert_{L^2(M)}   +  \Vert g_h \Vert_{L^2(M)}).
\end{equation}
Now recall that $x=(x',x_n)$ where $x'\in \xR^{n-1}$. Let   $r =1$ if $n=2$, $r=2$ if $n \geq 3. $ Then $H^{1+\eps}(\xR^r) \subset L^\infty(\xR^r).$ Set $x' =(y,z)\in \xR^r \times \xR^{n-1-r}.$ We can write
\begin{align*}
 \Vert v_h \Vert_{L^2(B)} &\leq \Big( (\alpha h^\mez)^r \int \sup_{y\in \xR^r}\vert v_h(y,z,x_n)\vert^2 dz dx_n\Big)^\mez \leq (\alpha h^\mez)^\frac{r}{2} \Big(\int  \Vert v_h(\cdot,z,x_n)\Vert_{H^{1+ \eps}(\xR^r)}^2 dz dx_n\Big)^\mez\\
  &\leq C(\alpha h^\mez)^\frac{r}{2} \Vert v_h \Vert_{H^{1+ \eps}(\xR^n)} \leq  C\alpha ^\frac{r}{2} h^{\frac{r}{4} - \eps}\frac{1}{h} (h  \Vert\psi_h \Vert_{L^2(M)}   +  \Vert g_h \Vert_{L^2(M)}).
  \end{align*}
and since $\frac{r}{4}-\eps >0$ we obtain eventually
\begin{equation}\label{vh2}
\Vert v_h \Vert_{L^2(B)} \leq C\alpha ^\sigma( \Vert\psi_h \Vert_{L^2(M)}   +  \frac{1}{h}\Vert g_h \Vert_{L^2(M)})
\end{equation}
where $\sigma = \mez $ if $n=2$, $\sigma= 1$ if $n\geq 3.$

 Now let us consider $w_h.$ First of all we have
\begin{equation}\label{psiuh}
  (h^2 P +1) w_h = \Psi(hD)\widetilde{F}_{h} + \big[h^2P, \Psi(hD) \big] u_h=: G_{2h}
  \end{equation}
with, as above
\begin{equation}\label{Gh}
\Vert G_{2h} \Vert_{L^2(\xR^n)} \leq C(h  \Vert \psi_h \Vert_{L^2(M)}   +  \Vert g_h \Vert_{L^2(M)}).
\end{equation}
We notice that the semi classical principal symbol $q$ of the operator $Q =:     h^2 P + 1$ satisfies the following property
\begin{equation}\label{type-princ}
   \text{on the set }  \{(x,\xi)\in T^*(\xR^n): q(x, \xi) =0\}  \text{ we have } \frac{\partial q}{ \partial \xi} \neq 0. 
\end{equation}
Since $\mathcal{K}:=K \times \supp \Psi$ is a compact subset ot $T^*(\xR^n)$  we can   find a finite number of  subsets ot $T^*(\xR^n), \mathcal{V}_1, \ldots \mathcal{V}_N$ such that $\mathcal{K} \subset \cup_{j=1}^N \mathcal{V}$ and in which
\begin{equation}\label{vois}
\begin{aligned}
&(i) \quad \text{either} \quad \vert q(x,\xi)\vert \geq c_0>0\\
&(ii) \quad \text{or} \quad q(x,\xi) = e(x,\xi)(\xi_l + a(x, \xi')),  \quad a \text{ real}, \quad  e(x,\xi) \neq 0.
\end{aligned}
 \end{equation}
Then we can find  $(\zeta_j)_{j=1,\ldots,N}$ such that
 $$ \zeta_j \in C_0^\infty(\mathcal{V}_j), \quad \text{and} \quad \sum_{j=1}^N \zeta_j =1\quad \text{in a neighborhood of } \mathcal{K}.$$
Therefore we can write
\begin{equation}\label{dec-wh}
 \Psi(hD) u_h = w_h  = \sum_{j=1}^N \zeta_j(x, hD)w_h.
 \end{equation}
It is sufficient to bound each term so we shall skip the index $j.$
 
 {\bf case\, 1}. In $\mathcal{V} $ we have $\vert q(x,\xi)\vert \geq c_0>0.$
 
 In that case the symbol $ a= \frac{\zeta}{q}$ belongs to $S^0(\xR^n \times \xR^n).$ By the semi classical symbolic calculus and \eqref{psiuh} we can write
 \begin{align*}
 \zeta(x,hD) w_ h&= \zeta (x,hD)\Psi(hD)u_h = a(x,hD)Q(x,hD)\Psi(hD)u_h + R_h u_h \\
 &= a(x,hD)G_{2h} + R_hu_h 
 \end{align*}
where
  $$\Vert R_h u_h \Vert_{L^2(\xR^n)} \leq Ch \Vert   u_h \Vert_{L^2(\xR^n)}. $$
    It follows from \eqref{Gh} that
\begin{equation}
\sum_{k=0}^2\Vert (h \nabla)^k \zeta (x,hD) w_ h \Vert_{L^2(\xR^n)} \leq C(h  \Vert \psi_h \Vert_{L^2(M)}   +  \Vert g_h \Vert_{L^2(M)}) 
\end{equation}
  so we see that $ \zeta (x,hD) w_ h$ satisfies the same estimate \eqref{vh} as $v_h$. Therefore the same argument as before leads to
  \begin{equation}\label{zone-ell}
\Vert \zeta (x,hD) w_ h\Vert_{L^2(B)} \leq C\alpha ^\sigma( \Vert\psi_h \Vert_{L^2(M)}   +  \frac{1}{h}\Vert g_h \Vert_{L^2(M)}),
\end{equation}
where $\sigma = \mez $ if $n=2$, $\sigma= 1$ if $n\geq 3.$

{\bf case\, 2}. In $\mathcal{V} $ we have $q(x,\xi) = e(x,\xi)(\xi_l + a(x,\eta)),$   $a$   real, $\vert e(x,\xi)\vert \geq c_0>0.$
$$ l\in\{1,\ldots,  n-1\}, \eta =(\xi_1, \ldots, \xi_{l-1},\xi_{l+1},\ldots, \xi_n), \quad e\in S^0, \quad \vert e(x,\xi)\vert \geq c_0>0.$$
Let us set $x_l = t, x = (x_1, \ldots, x_{l-1},x_{l+1},\ldots, x_n).$  Recall (see \eqref{somme}) that $B_{\alpha,h}\subset \{(t,x):\vert t \vert \leq \alpha h^\mez\}.$

Using the symbolic calculus and \eqref{Gh}  we see easily that  
$$   \big( ih \frac{\partial }{\partial t} + a(t,x, hD_x)\big)  \zeta (x,hD) w_ h   = G_{3h},$$
where 
\begin{equation}\label{G3h}
 \Vert G_{3h} \Vert_{L^2(\xR^n)} \leq C(h  \Vert \psi_h \Vert_{L^2(M)}   +  \Vert g_h \Vert_{L^2(M)}).
\end{equation}
Since the symbol $a$ is real, computing $\frac{d}{dt} \Vert w(t, \cdot)\Vert^2_{L^2(\xR^{n-1})}$ we see easily that
$$\Vert  \zeta (x,hD) w_ h(t, \cdot)\Vert_{L^2(\xR^{n-1})} \leq C \int_{t_0}^t \Vert  \zeta (x,hD) w_ h(s, \cdot)\Vert_{L^2(\xR^{n-1})}\, ds + \frac{1}{h} \int_{t_0}^t \Vert G_{3h}(s, \cdot)\Vert_{L^2(\xR^{n-1})}\, ds.$$
Now since $\vert t \vert \leq \alpha h^\mez, \vert t_0 \vert \leq \alpha h^\mez$ using the Cauchy Schwarz inequality,  \eqref{G3h}  and the Gronwall inequality we obtain
 $$ \Vert  \zeta (x,hD) w_ h(t, \cdot)\Vert_{L^2(\xR^{n-1})} \leq C\alpha^\mez h^{\frac{1}{4}} (\Vert \psi_h \Vert _{L^2(M)}   + \frac{1}{h} \Vert g_h \Vert_{L^2(M)}).$$
 It follows that
\begin{equation}\label{zone-energ}
\Vert  \zeta (x,hD) w_ h \Vert_{L^2(B_{\alpha,h})} \leq C\alpha h^{\frac{1}{2}} (\Vert \psi_h \Vert _{L^2(M)} + \frac{1}{h} \Vert g_h \Vert_{L^2(M)}).
 \end{equation}

{\bf case\, 3}. In $\mathcal{V} $ we have $q(x,\xi) = e(x,\xi)(\xi_n + a(x,\xi')),$   $a$    real, $\vert e(x,\xi)\vert \geq c_0>0.$

Since $B_{\alpha,h} = \{x=(x',x_n)\in \xR^{n-1}\times \xR: \vert x' \vert \leq \alpha h^\mez, \vert x_n \vert \leq c_0\}$ we cannot use the same argument as in case 2. Instead we shall use   Strichartz estimates proved in \cite[Section 2.4]{BuGeTz04-3} and  \cite{KoTaZw07} (see also~\cite{Zw12}). First of all, as before we see that 
$$   \big( ih \frac{\partial }{\partial t} + a(x, hD_{x'})\big)  \zeta (x,hD) w_ h   = G_{4h} $$
where $t = x_n$ and $G_{4h}$ satisfies \eqref{G3h}.

Assume first $n \geq 4.$ It is proved in the above works  that with   $I = \{\vert t \vert \leq c_0\}$ one has
  \begin{equation}\label{stri}
\Vert \zeta (x,hD) w_ h \Vert_{L_{t}^2 ( I, L_{x'}^r(\xR^{n-1})} \leq Ch^{-\frac{1}{2}} \frac{1}{h} \Vert G_{4h}\Vert_{L_t^1(I,  L_{x'}^2(\xR^{n-1}))}, \quad r = \frac{2(n-1)}{n-3} .
\end{equation}
 Now set $B' = \{x'\in \xR^{n-1}: \vert x' \vert \leq \alpha h^\mez\}. $
Using the H\" older inequality we obtain
$$\Vert \zeta (x,hD) w_ h(t, \cdot) \Vert_{L^2(B')} \leq  C \alpha h^\mez  \Vert \zeta (x,hD) w_ h(t, \cdot) \Vert_{L^r(\xR^{n-1})} $$
which implies, using \eqref{stri} and \eqref{G3h} that
 \begin{equation}\label{zone-stri1}
\Vert \zeta (x,hD) w_ h(t, \cdot) \Vert_{L^2(B_{\alpha,h})} \leq  C\alpha  (\Vert \psi_h \Vert _{L^2(M)}   + \frac{1}{h} \Vert g_h \Vert_{L^2(M)}).
\end{equation}
When $n =3$ the Strichartz estimate \eqref{stri} does not hold but we have the weaker ones, with $\frac 1 q + \frac 2 r =1$, $r< + \infty$
  \begin{equation}\label{stri2}
\Vert \zeta (x,hD) w_ h \Vert_{L_{t}^{q} ( L_{x'}^r(\xR^2))} \leq C_rh^{-(\frac{1}{2}- \frac 1 r)}  \frac{1}{h} \Vert G_{4h}\Vert_{L_t^1( L_{x'}^2(\xR^{2}))} 
\end{equation}
where (see~\cite{LiLo01}) 
$$ C_r \leq C r^{1/2}.$$
Then the H\" older inequality gives
$$\Vert \zeta (x,hD) w_ h(t, \cdot) \Vert_{L^2(B')} \leq  C_r (\alpha h^\mez)^{2 (\frac 1 2 - \frac 1 r)}  \Vert \zeta (x,hD) w_ h(t, \cdot)\Vert_{L^r} $$
and therefore
\begin{equation}\label{zone-stri2}
\Vert \zeta (x,hD) w_ h(t, \cdot) \Vert_{L^2(B_{\alpha,h})} \leq  C r^{1/2}\alpha^{\frac 1 2 - \frac 1 r}  (\Vert \psi_h \Vert _{L^2(M)}   + \frac{1}{h} \Vert g_h \Vert_{L^2(M)}).
\end{equation}
Optimizing with respect to $r<+ \infty $ leads to the choice $r = 4 \log ( \alpha^{-1})$, which gives a $\sqrt{\log( \alpha^{-1})} $ loss in the final estimate.
In the case $n =2$ we have instead the estimate
 \begin{equation*} 
\Vert \zeta (x,hD) w_ h \Vert_{L_{t}^4 ( I, L_{x'}^\infty(\xR))} \leq Ch^{-\frac{1}{4}} \frac{1}{h} \Vert G_{4h}\Vert_{L_t^1(I,  L_{x'}^2(\xR)) }.
\end{equation*}
which gives eventually
 \begin{equation}\label{zone-stri3}
\Vert \zeta (x,hD) w_ h(t, \cdot) \Vert_{L^2(B_{\alpha,h})} \leq  C\alpha^\mez  (\Vert \psi_h \Vert _{L^2(M)}   + \frac{1}{h} \Vert g_h \Vert_{L^2(M)}).
\end{equation}
 Then the conclusion in Proposition \ref{quasi-mode} follows from \eqref{somme}, \eqref{uh}, \eqref{vh2}, \eqref{zone-ell}, \eqref{zone-energ}, \eqref{zone-stri1}, \eqref{zone-stri2}, \eqref{zone-stri3}.
 \subsection{The general case: submanifolds of dimension $k$, $1 \leq k \leq n-1$}

 The   Laplace-Beltrami  operator $-\Delta_g$ with domain $D = \{u \in L^2(M): \Delta_g u \in L^2(M)\}$ has a discrete spectrum which can be written
   $$0 = \lambda_0^2< \lambda_1^2<\cdots < \lambda_j ^2\cdots \to + \infty$$
  where $  \lambda_j >0 , j \geq 1$ and     $-\Delta_g \varphi =  \lambda_j ^2 \varphi.$ 
   
   Moreover we can write $L^2(M) = \oplus_{j=0}^{+ \infty} \mathcal{H}_j,$ where $\mathcal{H}_j$ is the subspace of eigenvectors associated to the eigenvalue $\lambda_j ^2$ and  $\mathcal{H}_j \perp \mathcal{H}_k$ if $j \neq k.$

For $\lambda \geq 0$   we define  the spectral projector  $\Pi_\lambda:L^2(M) \to L^2(M)$ by
 \begin{equation}\label{def-proj}
    L^2(M) \ni f = \sum_{j\in \xN} \varphi_j,   \mapsto     \Pi_\lambda f = \sum_{j \in \Lambda_\lambda} \varphi_j, \quad \Lambda_\lambda= \{j \in \xN: \lambda_j \in [\lambda, \lambda+1)\}.
    \end{equation}
 Then   $\Pi_\lambda$ is self adjoint and $\Pi^2_\lambda =\Pi_\lambda$.

 Theorem \ref{quasi-mode} will be a consequence of the following one. Recall  $ N_{\smash{\alpha h^{1/2}}}$ has been defined in   \eqref{N}.

 \begin{prop}\label{proj-spec}
 There exist $C>0, h_0>0$ such that for every $h \leq h_0$ and every $\alpha \in (0,1)$
 \begin{equation}\label{est-proj}
 \Vert \Pi_\lambda u \Vert_{L^2( N_{\smash{\alpha h^{1/2}}})} \leq C \alpha^\sigma\Vert u \Vert_{L^2(M)}, \quad \lambda = \frac{1}{h},
\end{equation}
for every $u \in L^2(M),$  Here  $\sigma = 1$ if $k \leq n-3$, $\sigma = 1^- $ if $k = n-2, \sigma = \mez $ if $k = n-1.$
\end{prop}
Here, as before, $1^-$ means that we have an estimate by $C \alpha \vert \log(\alpha) \vert.$
 \subsubsection{Proof of Theorem \ref{quasi-mode} assuming Proposition \ref{proj-spec}.}
  If $ \psi = \sum_{j \geq 0} \varphi_j$ we have $  g = (h^2\Delta_g+1) \psi = \sum_{j \geq 0}(h^2 \Delta_g +1) \varphi_j.$ Therefore  by orthogonality
 \begin{equation}\label{norme-g}
 \Vert g \Vert^2_{L^2(M)} =  \sum_{j \geq 0}\vert 1- h^2 \lambda_j^2 \vert^2 \Vert \varphi_j \Vert^2_{L^2(M)}.  
 \end{equation}
   Let $\eps_0$ be a fixed number in $]0,1[.$ With $N =[\eps_0 \lambda]$ we write
 $$ \psi = \sum_{k = -N}^N \Pi_{\lambda +k} \psi + R_N.$$
 Recall that $  \Pi_{\lambda +k} \psi= \sum_{j \in E_k} \varphi_j$, where $E_k = \{j \geq 0: \lambda_j \in [\lambda +k, \lambda +k +1[.$

 Assume   $\vert k \vert \geq 2.$ Since $\lambda +k \leq \lambda_j < \lambda +k +1 $  we have $ \vert \lambda_j - \lambda\vert \geq \mez \vert k \vert$ which implies that $  \vert \lambda^2_j - \lambda^2 \vert \geq \mez \vert k \vert \lambda.$  By orthogonality we have 
 \begin{align*}
  \Vert \Pi_{\lambda +k} \psi \Vert^2_{L^2(M)} &=  \sum_{j \in E_k} \Vert \varphi _j\Vert^2_{L^2(M)} =  \sum_{j\in E_k} \frac{1}{\vert \lambda^2_j - \lambda^2 \vert^2} \vert \lambda^2_j - \lambda^2 \vert^2\Vert \varphi_j\Vert^2_{L^2(M)}\\
  & \leq \frac{4}{\vert k \vert^2 \lambda^2} \sum_{j\in E_k}   \vert \lambda^2_j - \lambda^2 \vert^2\Vert \varphi_j\Vert^2_{L^2(M)} \leq \frac{4 \lambda^2}{\vert k \vert^2 } \sum_{j\in E_k}  \vert h^2\lambda^2_j - 1 \vert^2\Vert \varphi_j\Vert^2_{L^2(M)}.
\end{align*}
Since  $ \Pi_{\lambda +k}^2 =  \Pi_{\lambda +k}$, using Proposition \ref{proj-spec} and the above estimate we obtain 
\begin{align*}
 \Vert  \sum_{2 \leq \vert k \vert \leq N} \Pi_{\lambda +k} \psi  \Vert_{L^2( N_{\smash{\alpha h^{1/2}}})}& \leq  \sum_{2 \leq \vert k \vert \leq N}  \Vert  \Pi_{\lambda +k} \psi \Vert_{L^2( N_{\smash{\alpha h^{1/2}}})} \leq C \alpha^\sigma \sum_{2 \leq \vert k \vert \leq N}  \Vert  \Pi_{\lambda +k} \psi \Vert_{L^2(M)}\\
 & \leq 2C \alpha^\sigma\lambda \sum_{2 \leq \vert k \vert \leq N}\frac{ 1 }{\vert k \vert} \Big(\sum_{j\in E_k}  \vert h^2\lambda^2_j - 1 \vert^2\Vert \varphi_j\Vert^2_{L^2(M)}\Big)^\mez.
 \end{align*}
 Using Cauchy-Schwarz inequality, \eqref{norme-g} and the fact that the $E_k$ are  pairwise disjoints we obtain eventually
 \begin{equation}\label{est-somme}
 \Vert  \sum_{2 \leq \vert k \vert \leq N} \Pi_{\lambda +k} \psi  \Vert_{L^2( N_{\smash{\alpha h^{1/2}}})} \leq C \alpha^\sigma\frac{1}{h} \Vert g \Vert_{L^2(M)}.
\end{equation}
Now a direct application of Proposition \ref{proj-spec} shows that
\begin{equation}\label{k<2}
 \Vert  \sum_{ \vert k \vert \leq 1} \Pi_{\lambda +k} \psi  \Vert_{L^2( N_{\smash{\alpha h^{1/2}}})} \leq C \alpha^\sigma  \Vert \psi \Vert_{L^2(M)}.
 \end{equation}
Eventually let us consider the remainder $R_N$. We have 
   $$R_N = \sum_{j\in A} \varphi_j + \sum_{j\in B} \varphi_j, \quad A = \{j: \lambda_j \leq \lambda -N\}, \quad B = \{j: \lambda_j \geq \lambda +N +1\}.$$
 The two sums are estimated by the same way since in both cases we have $\vert \lambda_j - \lambda \vert \geq c \lambda$ thus $\vert \lambda^2_j - \lambda^2 \vert \geq c \lambda^2$. Then by orthogonality we write
 \begin{align*}
  \Vert \sum_{j\in A} \varphi_j \Vert^2_{L^2(M)} &=  \sum_{j\in A} \Vert \varphi_j\Vert^2_{L^2(M)} =  \sum_{j\in A} \frac{1}{\vert \lambda^2_j - \lambda^2 \vert^2} \vert \lambda^2_j - \lambda^2 \vert^2\Vert \varphi_j\Vert^2_{L^2(M)} \\
  & \leq \frac{C}{\lambda^4} \sum_{j\in A}   \vert \lambda^2_j - \lambda^2 \vert^2\Vert \varphi_j\Vert^2_{L^2(M)} \leq\sum_{j\in\xN}   \vert h^2\lambda^2_j - 1\vert^2\Vert \varphi_j\Vert^2_{L^2(M)} = \Vert g \Vert^2_{L^2(M)}.
   \end{align*}
   It follows that $ \Vert R_N\Vert_{L^2(M)} \leq C \Vert g \Vert_{L^2(M)}.$ Now $(h^2\Delta_g +1) R_N =   \sum_{j \in A \cup B} (1- h^2 \lambda_j^2) \varphi_j = :g_N$ and $\Vert g_N \Vert_{L^2(M)} \leq \Vert g  \Vert_{L^2(M)}$. So using Lemma \ref{A1} we obtain
   \begin{equation}\label{est-RN}
   \Vert R_N\Vert_{L^2( N_{\smash{\alpha h^{1/2}}})} \leq C\frac{\alpha^\sigma}{h} \Vert g \Vert_{L^2(M)}
\end{equation}
  where $\sigma = \mez$ if $k = n-1$, $\sigma = 1$ if $1 \leq k \leq n-2.$.
  
Then Theorem \ref{quasi-mode} follows from \eqref{est-somme}, \eqref{k<2} and \eqref{est-RN}.
 
\subsubsection{Proof of Proposition \ref{proj-spec}}
This proposition will be a consequence  of the following one.
\begin{prop}\label{proj-spec2}
Let $\chi \in C^\infty(\xR)$ be such that $\chi(0) \neq 0.$   There exist $C>0, h_0>0$ such that for every $h \leq h_0$,every $\alpha \in (0,1),$ and every $u \in L^2(M)$ we have
 \begin{equation}\label{est-projbis}
 \Vert \chi( \sqrt{- \Delta_g} - \lambda) u \Vert_{L^2( N_{\smash{\alpha h^{1/2}}})} \leq C \alpha^\sigma \Vert u \Vert_{L^2(M)}, \quad \lambda = \frac{1}{h} 
\end{equation}
where $ \chi( \sqrt{- \Delta_g} - \lambda) u = \sum_{j\in \xN}  \chi( \lambda_j - \lambda) \varphi_j $ if $u =\sum_{j\in \xN}  \varphi_j.$
  \end{prop}
\begin{proof}[Proof of Proposition \ref{proj-spec} assuming Proposition  \ref{proj-spec2}]
There exists $\delta = \frac{1}{N}>0$ and $c>0$ such that $\chi(t) \geq c $ for every $t \in [-\delta, \delta].$ Now let $E = \{j \in \xN: \lambda_j \in [\mu, \mu+ \delta)\}$ and set $\Pi^\delta_\mu u = \sum_{j\in E} \varphi_j.$ On $E$ we have $\chi(\lambda_j - \mu) \geq c>0$  therefore we can write
$$1_E(j) = \chi(\lambda_j-\mu) \frac{1_E(j)}{\chi(\lambda_j-\mu)}.$$
 It follows that  
 $$\Pi^\delta_\mu u = \chi( \sqrt{- \Delta_g} - \lambda) \circ Ru$$
 where $R$ is continuous  from $L^2(M)$ to itself with norm bounded by $\frac{1}{c}$. Therefore assuming  Proposition \ref{proj-spec2} we can write
 \begin{equation}\label{proj-delta}
  \Vert \Pi^\delta_\mu u  \Vert_{L^2( N_{\smash{\alpha h^{1/2}}})} \leq C \alpha^\sigma \Vert Ru \Vert_{L^2(M)} \leq  \frac{C}{c} \alpha^\sigma\Vert  u \Vert_{L^2(M)}.
  \end{equation}
 where the constants in the right are independent of $\mu.$ Now since 
 $$\{j: \lambda_j \in [\lambda, \lambda +1)\} = \cup_{k =0}^{N-1} \{j: \lambda_j \in [\lambda + k \delta, \lambda + (k+1) \delta)\}$$ 
 where the union is disjoint, one can write
  $\Pi_\lambda u = \sum_{k = 0}^{N-1} \Pi^\delta_{\lambda + k \delta}.$ 
 It follows from \eqref{proj-delta} that
 $$ \Vert \Pi _\lambda u  \Vert_{L^2( N_{\smash{\alpha h^{1/2}}})} \leq C'   \alpha^\sigma \Vert  u \Vert_{L^2(M)}$$
 which proves Proposition \ref{proj-spec}.
\end{proof}
It remains to prove Proposition \ref{proj-spec2}. 
Until the end of this section  $\sigma $ will be a real number such that 
$$\sigma = 1 \, \text{  if }  k \leq n-3,  \quad \sigma = 1- \eps\,  ( \eps>0) \, \text{  if }   k = n-2, \quad \sigma = \mez  \, \text{  if } k = n-1.$$
        As before for every $p \in \Sigma^k$ one can find an open neighborhood $U_p$ of $p$ in $M$, a neighborhood $B_0$ of the origin in $\xR^n$ a diffeomorphism $\theta$ from $U_p $ to $B_0$ such that
\begin{equation}\label{N-alpha}
\begin{aligned}
&(i)\quad \theta (U_p \cap \Sigma^k) = \{x =(x_a,x_b) \in (\xR^{k}\times \xR^{n-k}) \cap B_0: x_b=0\}\\
&(ii) \quad \theta( U_p \cap  N_{\smash{\alpha h^{1/2}}}) \subset B_{\alpha,h}=: \{x\in B_0: \vert x_b \vert \leq \alpha h^\mez\}.
\end{aligned}
 \end{equation}
 Now $\Sigma^k $ and $ N_{\smash{\alpha h^{1/2}}}$ for $h$ small, are covered by a finite number of such open neighborhoods i.e. $N_{\alpha  h^\mez}\subset \cup_{j=1}^{n_0} U_{p_j}.$ We take a partition of unity relative to this covering i.e. $(\zeta_j) \in C^\infty(M)$ with $\supp \zeta_j \in U_{p_j}$ and $\sum_{j=1}^{n_0} \zeta_j = 1$ in a fixed neighborhood $ \mathcal{O}$ of $\Sigma^k$ containing $ N_{\smash{\alpha h^{1/2}}}$. For $p \in \mathcal{O}$ we can therefore write
 $$  \chi( \sqrt{- \Delta_g} - \lambda)u(p) = \sum_{j = 1}^{n_0}  \chi( \sqrt{- \Delta_g} - \lambda) (\zeta_ju)(p).$$
Our aim being to bound each term of the right hand side, we shall   skip the index $j$ in what follows. Moreover we shall set for convenience
$$ \chi_\lambda  =:  \chi( \sqrt{- \Delta_g} - \lambda)  $$
 
 We shall use some  results in [BGT] from which we quote the following ones.
 \begin{thm}[\cite{BuGeTz05} Theorem 4]
 There exists  $\chi \in \mathcal{S}(\xR)$ such that $\chi(0) =1$ and for any $p_0 \in \Sigma^k$ there  a diffeomorphism $ \theta$ as above, open sets $W \subset V = \{x \in \xR^n: \vert x \vert \leq \eps_0\}$, a smooth function $a:W_x \times V_y \times \xR^+_\lambda \to \xC$ supported in the set
 $$\{(x,y) \in W \times V: \vert x \vert \leq c_0 \eps \leq c_1 \eps \leq \vert y \vert \leq c_2 \eps \ll 1\}$$
satisfying
 $$ \forall \alpha \in \xN^{2n}, \exists C_\alpha>0: \forall \lambda \geq 0, \vert \partial_{x,y}^\alpha a (x,y, \lambda) \vert \leq C_\alpha,$$
 an operator $\mathcal{R}_\lambda: L^2(M) \to L^\infty(M)$ satisfying
 $$\Vert\mathcal{R}_\lambda u \Vert_{L^\infty(M)} \leq C \Vert  u \Vert_{L^2(M)},$$
 such that  for every $ x \in U =: W \cap \{x: \vert x \vert \leq c \eps\},$ setting $\widetilde{u} = \zeta u\circ  \theta^{-1}$ we have
 \begin{equation}\label{theoBGT}
  \chi_\lambda(\zeta u)(\theta^{-1}(x))=  \lambda^{\frac{n-1}{2}} \int_{y \in V} e^{i \lambda \psi(x,y)} a(x,y, \lambda) \widetilde{u}(y) \, dy + (\mathcal{R}_ \lambda (\zeta u))(\theta^{-1}(x))
  \end{equation}
  where $\psi(x,y) = - d_g((\theta^{-1}(x)),(\theta^{-1}(y)))$ is the geodesic distance on $M$ between $\theta^{-1}(x)$ and $\theta^{-1}(y)$. Furthermore the symbol $a$ is real non negative,  does not vanish for $\vert x \vert \leq c \eps$ and $d_g((\theta^{-1}(x)),(\theta^{-1}(y)))\in [c_3 \eps, c_4 \eps].$
\end{thm}
Let us set 
\begin{equation}\label{T-lambda}
\mathcal{T}_\lambda \widetilde{u}(x) = \int_{y \in V} e^{i \lambda \psi(x,y)} a(x,y, \lambda) \widetilde{u}(y) \, dy. 
\end{equation}
It follow from \eqref{theoBGT} that
\begin{equation}\label{est1}
\Vert   \chi_\lambda(\zeta u) \Vert_{L^2(N_{\alpha, h})} \leq \lambda^{\frac{n-1}{2}} \Vert \mathcal{T}_\lambda \widetilde{u}\Vert_{L^2(B_{\alpha,h})} + \Vert   \mathcal{R}_\lambda(\zeta u) \Vert_{L^2(N_{\alpha, h})}
\end{equation}
   Let us look to the contribution of $
\mathcal{R}_\lambda.$ Since (see \eqref{N-alpha}) the volume of $ N_{\smash{\alpha h^{1/2}}}$ is bounded by $C(\alpha h^\mez)^{n-k}$ we can  write
  $$\Vert\mathcal{R}_\lambda (\zeta u) \Vert_{L^2( N_{\smash{\alpha h^{1/2}}})} \leq C(\alpha h^\mez)^{\frac{n-k}{2}}\Vert\mathcal{R}_\lambda(\zeta u) \Vert_{L^\infty(M)} \leq C(\alpha h^\mez)^{\frac{n-k}{2}}\Vert  u \Vert_{L^2(M)}.$$ 
  If $k = n-1$ we have $\alpha^\frac{n-k}{2} = \alpha^\mez$ and if $1 \leq k \leq n-2$ we have $ \alpha^\frac{n-k}{2} \leq \alpha.$ Therefore we get
  \begin{equation}\label{est-Rlambda}
  \Vert\mathcal{R}_\lambda (\zeta u)  \Vert_{L^2( N_{\smash{\alpha h^{1/2}}})} \leq C \alpha^\sigma\Vert  u \Vert_{L^2(M)}.
\end{equation}
According to \eqref{est1}  Proposition \ref{proj-spec2} will be a consequence of the following result.
\begin{prop}\label{est-T}
There exists positive constants $C, \lambda_0$ such that
\begin{equation}\label{est-T1} 
 \lambda^{\frac{n-1}{2}} \Vert \mathcal{T}_\lambda \widetilde{u}\Vert_{L^2(B_{\alpha,h})} \leq C \alpha^\sigma \Vert u \Vert_{L^2(M)} 
 \end{equation}
for every $\lambda \geq \lambda_0$ and every $u \in L^2(M)$.
\end{prop}
% This proposition will be a consequence of the following one.
 %\begin{prop}\label{est-TT*}
% Set $ \mathcal{S}_\lambda = \mathcal{T}_\lambda \mathcal{T}^*_\lambda.$ Assume that there exists positive constants $C, \lambda_0$ such that
%\begin{equation}\label{est-T2}
%\Vert  \mathcal{S}_\lambda v \Vert_{L^\infty(\xR^{n-k}, L^2(\xR^k))} \leq C h^{ \frac{n+k}{2}-1} \Vert v \Vert_{L^1( \xR^{n-k}, L^2(\xR^k))} 
   %\end{equation}
%for every $\lambda \geq \lambda_0$ 
%then \eqref{est-T1} holds.
%\end{prop}
%\begin{proof}[Proof of Proposition \ref{est-TT*}]
%Let $1_B$ the "indicatrice" fuction of the set $B_{\alpha, \lambda}.$ Then it follows from \eqref{est-T2} that
%\begin{align*}
%\Vert 1_B \mathcal{S}_\lambda 1_B v \Vert_{L^2(\xR^n)}& \leq (\alpha h^\mez)^{\frac{n-k}{2}}\Vert   \mathcal{S}_\lambda 1_B v \Vert_{L^\infty(\xR^{n-k},L^2(\xR^k))}\\
%& \leq C (\alpha h^\mez)^{\frac{n-k}{2}}h^{  \frac{n+k}{2}-1} \Vert 1_B v \Vert_{L^1( \xR^{n-k}, L^2(\xR^k))} \leq C (\alpha h^\mez)^{n-k}h^{  \frac{n+k}{2}-1} \Vert v \Vert_{L^2(\xR^n)}\\
%& \leq C\alpha ^{n-k}h^{n-1} \leq C \alpha ^{2 \sigma}h^{n-1} \Vert v \Vert_{L^2(\xR^n)}.
% \end{align*}
%This estimate implies as usual that
%$$ \Vert1_B \mathcal{T}_ \lambda \widetilde{u} \Vert_{L^2(\xR^n)} \leq C \alpha^\sigma h^{\frac{n-1}{2}} \Vert \widetilde{u} \Vert_{L^2(\xR^n)}\leq  C \alpha^\sigma h^{\frac{n-1}{2}} \Vert {u} \Vert_{L^2(M)}$$
%which is precisely \eqref{est-T1}.
%\end{proof}
\begin{proof}[Proof of Proposition \ref{est-T}]
Set $\mathcal{S}_\lambda = \mathcal{T}_\lambda \mathcal{T}_\lambda^*$ and  denote by $1_B$ the indicator function of the set $B_{\alpha,h}.$ By the usual trick  \eqref{est-T1} will be a consequence of the following estimate.
\begin{equation}\label{est-S}
\Vert 1_B \mathcal{S}_\lambda 1_B v \Vert_{L^2(\xR^n)} \leq C h^{n-1} \alpha^{2 \sigma} \Vert v \Vert_{L^2(\xR^n)}, \quad h = \frac{1}{\lambda}.
\end{equation}
  Let  $\mathcal{K}_\lambda(x,x')$  be the   kernel of $\mathcal{S}_\lambda$. By
   \eqref{T-lambda} it is given by   
\begin{equation}\label{K-lambda}
\mathcal{K}_\lambda(x,x') = \int e^{i \lambda [\psi(x,y) - \psi(x',y)]}a(x,y, \lambda) \overline{a}(x',y, \lambda) \, dy.
\end{equation}
 We shall decompose 
\begin{equation}\label{dec-K}
\left\{
\begin{aligned}
 &\mathcal{K}_\lambda = \mathcal{K}^1_\lambda + \mathcal{K}^2_\lambda,\\
&\mathcal{K}^1_\lambda = 1_{\{\vert x-x' \vert \leq \frac{1}{\lambda}\}
} \mathcal{K}_\lambda,  \quad   \mathcal{K}^2_\lambda = 1_{\{\frac{1}{\lambda} < \vert x-x' \vert \leq  \eps\}} \mathcal{K}_\lambda,\\
& \mathcal{S}_\lambda  = \sum_{j=1}^2 \mathcal{S}_\lambda^j, \quad \mathcal{S}_\lambda^j \widetilde{u}(x)= \int \mathcal{K}^j_\lambda(x,x') \widetilde{u}(x')\,dx'
\end{aligned}
\right.
\end{equation}
and  treat separately each piece.

 \subsubsection{Estimate of $\mathcal{S}_\lambda^1$}
 
 When $\vert x-x' \vert \leq \frac{1}{\lambda}$ the kernel $\mathcal{K}_\lambda $ is uniformly bounded. Therefore $ \vert \mathcal{K}^1_\lambda  \vert \leq C 1_{\{\vert x-x' \vert \leq \frac{1}{\lambda}\}}, $ so by Schur lemma we have
 \begin{equation*} 
   \Vert \mathcal{S}_\lambda^1v \Vert_{L^2(\xR^n)} \leq C h^n\Vert v \Vert_{L^2(\xR^n)}.
 \end{equation*}
 Therefore
 \begin{equation}\label{S1-1}
 \Vert 1_B \mathcal{S}_\lambda^1 1_B v \Vert_{L^2(\xR^n)} \leq  C h h^{n-1} \Vert v \Vert_{L^2(\xR^n)}.
\end{equation}
On the other hand writing $x = (x_a, x_b), x' = (x'_a, x'_b)$ we have
$$ \Vert \mathcal{S}_\lambda^1v(\cdot, x_b) \Vert_{L^2(\xR^k)} \leq C\int_{\xR^{n-k}} 1_{\{\vert x_b -x'_b \vert \leq h\}}\lA \int_{\xR^k} 1_{\{\vert x_a -x'_a \vert \leq h\}} v(x'_a,x'_b)\, dx'_a\rA_{L^2(\xR^k)}  \, dx'_b.$$
 Again by Schur lemma we get
  $$ \Vert \mathcal{S}_\lambda^1v  \Vert_{L^\infty(\xR^{n-k},L^2(\xR^k))} \leq C h^k    \Vert v \Vert_{L^1(\xR^{n-k}, L^2(\xR^k))}.$$
 We deduce that
 $$
  \Vert 1_B \mathcal{S}_\lambda^1 1_B v \Vert_{L^2(\xR^n)} \leq  C (\alpha h^\mez)^{n-k} h^k \Vert v \Vert_{L^2(\xR^n)}.
$$
This estimate can be rewritten as 
\begin{equation}\label{S1-2}  
   \Vert 1_B \mathcal{S}_\lambda^1 1_B v \Vert_{L^2(\xR^n)} \leq  C \alpha^{2 \sigma}  \alpha^{n-k-2 \sigma}h^{\frac {n-k}{2} + k}   \Vert v \Vert_{L^2(\xR^n)}. 
\end{equation}  
  Now if $ h^\mez \leq \alpha$ we use \eqref{S1-1} and we obtain 
  $$  \Vert 1_B \mathcal{S}_\lambda^1 1_B v \Vert_{L^2(\xR^n)} \leq  C \alpha^{2}  h^{n-1}\Vert v \Vert_{L^2(\xR^n}.$$
If $\alpha \leq h^\mez$ we use instead \eqref{S1-2}. Since $n-k-2\sigma \geq 0$ we can write
\begin{align*}
  \Vert 1_B \mathcal{S}_\lambda^1 1_B v \Vert_{L^2(\xR^n)}  &\leq C \alpha^{2 \sigma} h^{\mez(n-k-2 \sigma) + \mez(n-k) +k}  \Vert v \Vert_{L^2(\xR^n)}= C \alpha^{2 \sigma} h^{ n-  \sigma }  \Vert v \Vert_{L^2(\xR^n)}\\
  & \leq C \alpha^{2 \sigma} h^{ n-  1 }  \Vert v \Vert_{L^2(\xR^n)}.
  \end{align*}
 Therefore  in all cases we have 
 \begin{equation}\label{S1-3}
  \Vert 1_B \mathcal{S}_\lambda^1 1_B v \Vert_{L^2(\xR^n)} \leq C \alpha^{2 \sigma} h^{ n-  1 }  \Vert v \Vert_{L^2(\xR^n)}.
\end{equation}
  
   To deal with the  other  regime  we need the   description of the kernel $\mathcal{K} $ given in \cite{BuGeTz05}.
\begin{lemm}  [\cite{BuGeTz05} Lemma 6.1] \label{BGT}
There exists $\eps \ll1,( a_p^{\pm},b_p)_{p\in \xN} \in {C^\infty(\xR^n \times \xR^n \times \xR)}$ such that for $\vert x-x' \vert \gtrsim \lambda^{-1}$ and any $N \in \xN^*$ we have
$$\mathcal{K}_\lambda(x,x') = \sum_{\pm} \sum_{p =0}^{N-1} \frac{e^{  \pm i \lambda \widetilde{\psi}(x,x')}}{(\lambda \vert x-x'\vert)^{\frac{n-1}{2} + p}} a_p^{\pm}(x,x',\lambda) + b_N(x,x', \lambda)$$
where   $\widetilde{\psi}(x,x')$ is the geodesic distance between the points $\theta^{-1}(x)$ and $\theta^{-1}(x').$  Moreover $a_p^\pm$ are real, have supports of size $\mathcal{O}(\eps)$ with respect to the two first variables and are uniformly bounded with respect to $\lambda.$ Finally
$$ \vert  b_N(x,x', \lambda) \vert \leq C_N ( \lambda \vert x-x' \vert)^{-(\frac{d-1}{2} +N)}.$$
\end{lemm}
    
\subsubsection{Estimate of $\mathcal{S}_\lambda^2$}
  
  We cut the set $\frac{1}{\lambda} \leq \vert x-x'\vert \leq \eps$ into pieces 
  $$ \vert x-x' \vert \sim 2^{-j}, \quad \frac{1}{\lambda} \leq  2^{-j} \leq \eps$$
  and we estimate the contribution of each term. According to Lemma \ref{BGT} we are lead to work with the operator
   $$A_j v(x) = \int k_ j(x,x', \lambda) v(x') \, dx' $$
   where
   \begin{equation}\label{kj}
  k_ j(x,x', \lambda)= (\lambda  2^{-j})^{- \frac{n-1}{2}}\chi_0(2^j(x-x')) e^{i\lambda \widetilde{\psi}(x,x')} \sum_{p=0}^{N-1} \lambda^{-p} a_p(x,x',\lambda). 
  \end{equation}
  Now there exists   $\chi \in C^\infty(\xR^n)$ such that $\supp \chi \subset \{x: \vert x \vert \leq 1\}, \chi(x) = 1 $ if $\vert x \vert \leq \mez$ and
  $$ \sum_{p\in \xZ^n} \chi(x-p) = 1, \forall x \in \xR^n.$$
 Following \cite{BuGeTz05} we   write
 \begin{equation}\label{kjpq}
 \begin{aligned}
 k_j(x,x',\lambda) &= \sum_{p,q \in \xZ^n} k_{j p q}(x,x',\lambda)\\
  k_{j p q}(x,x',\lambda)& = \chi(2^jx-p) k_j(x,x',\lambda)\chi(2^jx'-q) 
  \end{aligned}
 \end{equation}
and we denote by $A_{jpq}$ the operator with kernel $ k_{j p q}.$

Notice that the sum appearing in \eqref{kjpq} is to be taken only for  $\vert p-q \vert \leq 2. $   

We claim that   by quasi orthogonality in $L^2$ we have
\begin{equation}\label{Aj-Ajpq}
\Vert 1_B A_j 1_B \Vert _{L^2(\xR^n) \to L^2(\xR^n)} \leq C \sup_{\vert p-q \vert \leq 2}\Vert 1_B A_{jpq} 1_B \Vert _{L^2(\xR^n) \to L^2(\xR^n)}.
\end{equation}
Indeed let us forget $1_B$ which plays any role. We have
$$\Vert A_j v \Vert_{L^2(\xR^n)} = \sum_{\vert p-q \vert \leq 2}  \sum_{\vert p'-q' \vert \leq 2}\int A_{jpq}[\widetilde{ \chi}(2^j\cdot-q)v](x)  A_{jp'q'}[\widetilde{ \chi}(2^j\cdot-q')v](x)\, dx$$
where $\widetilde{\chi} \in C_0^\infty(\xR^n), \widetilde{\chi} = 1$ on the support of $\chi$ and $\sum_{p\in \xZ^n} [\widetilde{\chi} (x-p)]^2 \leq M, \forall x \in \xR^n.$ Due to the presence of $\chi(2^jx-p)$, $\chi(2^jx-p')$ ans $\chi_0(2^j(x-x')$ inside the above integral one must also have $\vert p-p' \vert \leq 2$ in the sum.
Therefore we are summing on the set  $E = \{(p,q,p',q'): \vert p-q \vert \leq 2, \vert p-p' \vert \leq 2, \vert p'-q' \vert \leq 2\}.$ We have
\begin{align*}
E \subset E_1 &= \{(p,q,p',q'): \vert p-q \vert \leq 2, \vert p' -q \vert \leq 4, \vert q' -q \vert \leq 6\},\\
E \subset E_2&= \{(p,q,p',q'): \vert p'-q' \vert \leq 2, \vert p-q' \vert \leq 4, \vert q  -q' \vert \leq 6\}.
\end{align*}
It follows from  the Cauchy-Schwarz inequality that $\Vert A_j v \Vert_{L^2(\xR^n)}$ can be bounded by
 $$   \Big( \sum_{E_1} \Vert A_{jpq}\Vert^2_{L^2 \to L^2} \Vert  \widetilde{ \chi}(2^j\cdot-q)v\Vert^2_{L^2(\xR^n)}\Big)^\mez \Big( \sum_{E_2} \Vert A_{jp'q'}\Vert^2_{L^2 \to L^2} \Vert  \widetilde{ \chi}(2^j\cdot-q')v\Vert^2_{L^2(\xR^n)}\Big)^\mez$$
and therefore by the choice of $\widetilde{\chi}$ by
 $C \sup_{\vert p-q \vert \leq 2}\Vert   A_{jpq}  \Vert^2 _{L^2(\xR^n) \to L^2(\xR^n)} \Vert v \Vert^2_{L^2(\xR^n)} $ which proves our claim.

Now let us consider the operator $Q_{jpq}$ defined by
\begin{equation}\label{Bjpq}
\begin{aligned}
Q_{jpq} v(X) &= \int_{\xR^n} \sigma_{jpq}(X,X',\lambda) v(X')\, dX'\\
  \sigma_{jpq}(X,X',\lambda) & = \chi(X-p) k_j(2^{-j}X, 2^{-j} X', \lambda) \chi(X'-q).
\end{aligned}
\end{equation}
Then by the change  of variables $ (x = 2^{-j}X, x' = 2^{-j}X') $ we can see easily that 
\begin{align}
&\Vert 1_{2^jB} Q_{jpq} 1_{2^jB} v \Vert_{L^2(\xR^n)} \leq K_j \Vert v \Vert_{L^2(\xR^n)} \quad \text{implies} \label{Qjpq} \\
& \Vert 1_{ B} A_{jpq} 1_{ B} v \Vert_{L^2(\xR^n)} \leq 2^{-jn} K_j \Vert v \Vert_{L^2(\xR^n)}. \label{Ajpq}
\end{align}
Setting
\begin{equation}\label{psij}
\mu_j = \lambda 2^{-j}, \quad \widetilde{\psi}_j(X,X') = 2^j  \widetilde{\psi}(2^{-j}X,2^{-j}X'),
\end{equation}
 we deduce from  \eqref{kj} and \eqref{Bjpq} we have
\begin{equation}\label{sigmajpq}
\begin{aligned}
\sigma_{jpq}(X,X',\lambda) =\mu_j^{-\frac{n-1}{2}} e^{i \mu_j \widetilde{\psi}_j(  X,  X')} &\chi(X-p) \chi(X-q)   \chi_0(X-X') \\ & \cdot \sum_{p=0}^{N-1}\lambda^{-p} a_p(2^{-j}X , 2^{-j}X', \lambda).
\end{aligned}
\end{equation}
We shall derive two estimates of the left hand side of \eqref{Qjpq}. On one hand  using  Theorem \ref{stein2} with $p = k-1$ we can  write,
\begin{align*}
  \Vert 1_{2^jB} Q_{jpq} 1_{2^jB} v \Vert_{L^2(\xR^n)} &\leq C  ( \alpha h^\mez 2^j)^{\frac{n-k}{2}}  \Vert    Q_{jpq} 1_{2^jB} v \Vert_{L^\infty(\xR^{n-k}_{x_b}\times \xR_{x_{a1}}, L^2 (\xR^{k-1}_{x_a'}))},\\
  &\leq C \mu_j^{-\frac{n-1}{2}}( \alpha h^\mez 2^j)^{\frac{n-k}{2}} \mu_j^{-\frac{k-1}{2}}  \Vert1_{2^jB} v \Vert_{L^1(\xR^{n-k}_{x_b}\times \xR_{x_{a1}}, L^2 (\xR^{k-1}_{x_a'}))},\\
  &\leq C \mu_j^{-\frac{n-1}{2}}( \alpha h^\mez 2^j)^{n-k} \mu_j^{-\frac{k-1}{2}} \Vert v \Vert_{L^2(\xR^n)}.
  \end{align*}
 We deduce from \eqref{Ajpq} and \eqref{Aj-Ajpq} that 
 \begin{equation}\label{Ajpq2}
 \Vert 1_{ B} A_{j} 1_{ B} v \Vert_{L^2(\xR^n)} \leq C h^{n-1} \alpha^{n-k}2^{j(\frac{n-k}{2}-1)}  \Vert v \Vert_{L^2(\xR^n)}.
\end{equation}
On the other hand using   Theorem \ref{stein1} with $p = n-1$ we can write
$$ \Vert 1_{2^jB} Q_{jpq} 1_{2^jB} v \Vert_{L^2(\xR^n)} \leq \Vert  Q_{jpq} 1_{2^jB} v \Vert_{L^2(\xR^n)}  \leq C \mu_j^{-\frac{n-1}{2}} \mu_j^{-\frac{n-1}{2}} \Vert v \Vert_{L^2(\xR^n)},$$ 
from which we deduce using \eqref{Ajpq} and \eqref{Aj-Ajpq} that
 \begin{equation}\label{Ajpq3}
   \Vert 1_{ B} A_{j} 1_{ B} v \Vert_{L^2(\xR^n)}  \leq C 2^{-jn} (2^j h)^{ n-1 }   \leq C h^{n-1} 2^{-j}.
  \end{equation}
Recall that we  have  $\mathcal{S}_\lambda^2 = \sum_{j\in E} A_j$ where $E = \{j: \frac{1}{\eps} \leq 2^j \leq \lambda\}.$  Then we write
\begin{equation}\label{S3}
\begin{aligned}
1_B S_\lambda^2 1_B v &= \sum_{ j \in E_1} 1_B A_j1_B v + \sum_{ j \in E_2} 1_B A_j1_B v =(1) + (2), \quad \text{where} \\
E_1 &=\{j: \frac{1}{\eps} \leq 2^j \leq \alpha^{-2}\}, \quad E_2= \{j: \alpha^{-2} \leq 2^j \leq \lambda\}.
\end{aligned}
\end{equation}
  To estimate the term $(1)$ we use \eqref{Ajpq2}. We obtain
  $$\Vert (1) \Vert_{L^2(\xR^n)} \leq C h^{n-1} \alpha^{n-k} \sum_{j \in E_1} 2^{j(\frac{n-k}{2} -1)} \Vert v \Vert_{L^2(\xR^n)}.$$
 Then we have three cases. 
 
 If $\frac{n-k}{2} -1 >0$ that is if $ k \leq n-3$ then 
 $$\Vert (1) \Vert_{L^2(\xR^n)} \leq C h^{n-1} \alpha^{n-k} \Big( \frac{1}{\alpha^2}\Big)^{\frac{n-k}{2} -1 }  \Vert v \Vert_{L^2(\xR^n)}\leq C h^{n-1} \alpha^2  \Vert v \Vert_{L^2(\xR^n)}.$$
  If $\frac{n-k}{2} -1  =0$ that is if $ k  = n-2$ then 
  $$\Vert (1) \Vert_{L^2(\xR^n)} \leq C  h^{n-1} \alpha^2  \text{ Log} (\alpha^{-1})  \Vert v \Vert_{L^2(\xR^n)} .$$
 If  $ k = n-1$ then
 $$\Vert (1) \Vert_{L^2(\xR^n)} \leq C h^{n-1} \alpha \sum_{j=0}^\infty 2^{-j}  \Vert v \Vert_{L^2(\xR^n)} \leq  C h^{n-1} \alpha \Vert v \Vert_{L^2(\xR^n)}.$$
  To estimate the term $(2)$ we use \eqref{Ajpq3}. We obtain
  $$ \Vert (2) \Vert_{L^2(\xR^n)} \leq C h^{n-1} \alpha^{2} \Vert v \Vert_{L^2(\xR^n)}.$$
  Using these estimates and \eqref{S3} we deduce
  \begin{equation}\label{estS3}
  \Vert 1_B S_\lambda^2 1_B v \Vert_{L^2( \xR^n)} \leq C \alpha^{2\sigma}  h^{n-1}\Vert v \Vert_{L^2(\xR^n)} 
\end{equation}
 where  $\sigma = 1$ if $k \leq n-3$, $\sigma = 1- \eps$ if $k = n-2, \sigma = \mez $ if $k = n-1.$
  \end{proof} 
    Gathering the estimates proved in \eqref{S1-3} and \eqref{estS3} we obtain \eqref{est-S} which proves Proposition \ref{est-T} and therefore Proposition \ref{proj-spec}.   The proof of  Theorem \ref{quasi-mode} is complete.
    
 \setcounter{section}{0}   
\renewcommand{\thesection}{\Alph{section}}
  \section{Some technical results}
  \subsection{A lemma}
   \begin{lemm}\label{A1}
   Let $w \in C^\infty(M)$ be a solution of the equation $(h^2\Delta_g +1)w = F$ Then
   $$  \Vert w \Vert_{L^2(  N_{\smash{\alpha h^{1/2}}})} \leq C \frac{\alpha^\gamma}{h} \big(\Vert F \Vert_{L^2(M)} + \Vert w \Vert_{L^2(M)} \big) $$
   where $\gamma = \mez$ if $k = n-1$, $\gamma = 1$ if $1 \leq k \leq n-2.$ \end{lemm}
\begin{proof}
Setting  $\Vert \nabla_g w \Vert_{L^2(M)}=  \Big(\int_Mg_p \big (\nabla_g w(p), \overline{\nabla_g w(p)} \big) \, dv_g(p)\Big)^\mez $ we deduce from Lemma \ref{energie} and from the equation that 
\begin{equation}\label{reg-H2}
\begin{aligned}
 &h\Vert \nabla_g w \Vert_{L^2(M)} \leq C \big(\Vert F \Vert_{L^2(M)} + \Vert w \Vert_{L^2(M)} \big),\\
 & h^2 \Vert \Delta_g w\Vert_{L^2(M)} \leq  \Vert F \Vert_{L^2(M)} + \Vert w \Vert_{L^2(M)}.
 \end{aligned}
 \end{equation}
 Now setting $ \widetilde{w}_j  = (\chi_jw)\circ \theta^{-1}$ we have
 \begin{equation}\label{est-w1}
  \Vert w \Vert_{L^2(  N_{\smash{\alpha h^{1/2}}})}  \leq \sum_{j=1}^{n_0} \Vert \chi_jw \Vert_{L^2(  N_{\smash{\alpha h^{1/2}}})} \leq C\sum_{j=1}^{n_0} \Vert \widetilde{w}_j \Vert_{L^2( B_{\alpha,h})}.
  \end{equation}
 For fixed $j \in \{1, \ldots, n_0\}$  we deduce from \eqref{reg-H2} that
 \begin{equation}\label{reg-H2'}
  h \Vert \widetilde{w}_j \Vert_{H^1(B_{\alpha,h})} + h^2 \Vert \widetilde{w}_j \Vert_{H^2(B_{\alpha,h})} \leq C \big(\Vert F \Vert_{L^2(M)} + \Vert w \Vert_{L^2(M)} \big),
  \end{equation}
 from which we deduce that for $\eps>0$ small
 \begin{equation}\label{reg-1+eps}
  h^{1+ \eps} \Vert \widetilde{w}_j \Vert_{H^{1+ \eps}(B_{\alpha,h})}  \leq C \big(\Vert F \Vert_{L^2(M)} + \Vert w \Vert_{L^2(M)} \big).
  \end{equation}
Using the Sobolev embeddings  $H^1(\xR) \subset L^\infty(\xR)$ and $H^{1+ \eps}(\xR^2) \subset L^\infty(\xR^2)$,  the fact that $B_{\alpha,h}\subset  \{x= (x_a,x_b) \in \xR^k \times \xR^{n-k}: \vert x_b \vert \leq \alpha h^\mez\}$ and  \eqref{reg-H2'},  \eqref{reg-1+eps} we obtain
\begin{align*}
& \Vert \widetilde{w}_j \Vert_{L^2( B_{\alpha,h})} \leq  (\alpha h^\mez)^\mez \Vert \widetilde{w}_j \Vert_{H^1(B_{\alpha,h})} \leq C  \frac{\alpha^\mez }{h} \big(\Vert F \Vert_{L^2(M)} + \Vert w \Vert_{L^2(M)} \big)  
, \quad \text{ if } k=n-1,\\
&  \Vert \widetilde{w}_j \Vert_{L^2( B_{\alpha,h})} \leq \alpha h^\mez \Vert \widetilde{w}_j \Vert_{H^{1+ \eps}(B_{\alpha,h})} \leq  C \frac{\alpha}{h} \big(\Vert F \Vert_{L^2(M)} + \Vert w \Vert_{L^2(M)} \big)  
,  \quad \text{ if } k\leq n-2.
 \end{align*}
  Lemma \ref{A1} follows then from \eqref{est-w1}.
\end{proof}

   \subsection{Stein's lemma}
   In this section we prove a version of Stein Lemma~\cite[Chap 9, Proposition 1.1]{St93}. 
For $\lambda>0$ we consider the operator
 \begin{equation}\label{Tlambda}
 T^\lambda u(\Xi) = \int_{\xR^n} e^{i\lambda\phi(X, \Xi)} a(X, \Xi, \lambda) u(X) \, dX
\end{equation}
 where $\phi :\xR^n \times \xR^n  \to \xR$ is a smooth real valued  phase and $a$   a smooth symbol. 
 
 We shall make the following assumptions.
 \begin{align*}
 &(\text{H}1) \,  \text{ there exists a compact } \mathcal{K} \subset \xR^n \times \xR^n  \text{ such that }
 \supp_{X, \Xi} a\subset \mathcal{K}, \quad \forall \lambda>0,\\
 &(\text{H}2)  \,  \text{ rank }  \Big(\frac{\partial^2 \phi}{\partial X_i \partial \Xi_j}(X, \Xi)\Big)_{1 \leq i,j \leq n} \geq  p \in \{1, \ldots, n\},  \forall (X, \Xi) \in \mathcal{K}.
\end{align*}
Our purpose is to prove the following result.
 \begin{thm}\label{stein1}
 Under the hypotheses (H1) and (H2) there exists   $C>0$ such that  $$
 \Vert T^\lambda u \Vert_{L^2(\xR^n)} \leq C \lambda^{-\frac{p}{2}} \Vert    u \Vert_{L^2(\xR^n)}
 $$ 
  for every $\lambda>0$ and all  $u \in L^2(\xR^n).$
 \end{thm}
 \begin{rema}
 We shall actually apply Theorem~\ref{stein1} for a family of phases $\phi_j$ and symbols $a_j$ converging in $C^\infty$ topology to a fixed phase $\phi$ and symbol $a$ and use that in such case the estimates are uniform with respect to the parameter $j$, which will be a consequence of the proof given below.
 \end{rema} 
Below we shall prove a slightly stronger result.
 
  First of all by the hypothesis (H1), using   partitions of unity, we may assume without loss of generality that with a small $\eps>0$ 
  $$\supp_{X, \Xi}a \subset V_{\rho_0} = \{(X,\Xi)\in \xR^n \times \xR^n: \vert X-X_0\vert + \vert \Xi - \Xi_0\vert \leq \eps\}, \quad \rho_0 = (X_0, \Xi_0).$$
  Moreover changing if necessary the orders of the variables we may assume that near $\rho_0$
  $$X = (x,y) \in \xR^p \times \xR^{n-p}, \quad \Xi = (\xi,\eta) \in \xR^p \times \xR^{n-p} $$ 
  and for all $(X,\Xi)\in   V_{\rho_0}$ the $p\times p$-matrix
\begin{equation}\label{eq.H}
 M_p(X, \Xi) = \Big(\frac{\partial^2 \phi}{\partial x_i \partial \xi_j}(X, \Xi)\Big)_{1 \leq i,j \leq p}
 \end{equation}
 is invertible with $\Vert  M_p(X, \Xi)^{-1} \Vert \leq c_0.$

 Then we have
 \begin{thm}\label{stein2}
 There exists a positive constant $C$ such that for every $\lambda>0$ we have
 $$
 \Vert T^\lambda u \Vert_{L^\infty(\xR^{n-p}_\eta,   L^2(\xR^p_\xi))} \leq C \lambda^{-\frac{p}{2}} \Vert    u \Vert_{L^1(\xR^{n-p}_y,   L^2(\xR^p_x))}
 $$ 
 for all $u\in L^1(\xR^{n-p}_y,   L^2(\xR^p_x)).$
\end{thm}
 Theorem \ref{stein1}   follows from Theorem \ref{stein2} using (H1).
  \begin{proof}[Proof of Theorem \ref{stein2}]
 It is an easy consequence of the proof of a proposition in section 1.1 Chapter IX in~\cite{St93}. Indeed let us set for $(y,\eta)\in \xR^{n-p} \times \xR^{n-p}$
 \begin{align}
 &\phi_{(y,\eta)}(x,\xi) = \phi(x,y,\xi,\eta), \quad a_{(y,\eta)}(x,\xi) = a(x,y,\xi,\eta), \quad u_{ y }(x) = u(x,y),\\
 & T^\lambda_{(y,\eta)} u_y(\xi) = \int_{\xR^p} e^{i\lambda  \phi_{(y,\eta)}(x,\xi)}a_{(y,\eta)}(x,\xi)  u_{ y }(x)\, dx.\label{local}
 \end{align}
 Then we have
 \begin{equation}\label{dec}
 T^\lambda u(\Xi) = \int_{\xR^{n-p}} T^\lambda_{(y,\eta)} u_y(\xi)\, dy.
 \end{equation}
 We claim that there exists $C>0$ such that for every $ (y,\eta)\in V_{(y_0,\eta_0)}$ we have
 \begin{equation}\label{est-stein}
  \Vert T^\lambda_{(y,\eta)} u_y \Vert_{L^2(\xR^p_\xi)} \leq C \lambda^{-\frac{p}{2}} \Vert   u_y \Vert_{L^2(\xR_x^p)} \quad \forall \lambda>0.
  \end{equation}
Assuming for a moment that \eqref{est-stein} is proved we obtain 
     $$\Vert T^\lambda u(\cdot, \eta)\Vert_{L^2(\xR^p_\xi)} \leq   \int_{\xR^{n-p}} \Vert T^\lambda_{(y,\eta)} u_y \Vert_{L^2(\xR^n_\xi)} \, dy \leq C \lambda^{-\frac{p}{2}} \int_{\xR^{n-p}} \Vert u(\cdot,y)\Vert_{L^2(\xR^p_x)}\, dy$$
 which implies immediately the conclusion of Theorem \ref{stein1}.
 
 The claim \eqref{est-stein} follows immediately from the proof of proposition in ¤1.1 Chapter IX in [Stein]. However, for the convenience of the reader,  we shall give it here. 
 
 For simplicity we shall skip the subscript $(y, \eta),$ keeping in mind the uniformity, with respect to $(y, \eta) \in V_{(y_0, \eta_0)}$,  of the constants in the estimates. Therefore we  set 
 $$S_\lambda = T^\lambda_{(y,\eta)},\quad \phi_{(y,\eta)}= \psi,  \quad b = a_{(y,\eta)}.$$
 It follows from \eqref{eq.H} that the matrix
 $$N(x, \xi) = \Big( \frac{\partial^2 \psi}{\partial x_i \partial \xi_j}(x,\xi) \Big)_{1 \leq i,j \leq p}$$
 is invertible and $\Vert N(x, \xi)^{-1} \Vert \leq c_0$ where $c_0$ is independent of $(y,\eta).$
  Now by the usual trick the estimate \eqref{est-stein} is satisfied if and only if we have
  \begin{equation}\label{trick}
   \Vert S_\lambda S_\lambda^* f\Vert_{L^2(\xR^p)} \leq C \lambda^{-p} \Vert f \Vert_{L^2(\xR^p)}
   \end{equation}
  with $C$ independent of $(y, \eta).$ It is easy to see that 
  \begin{equation}\label{SS*}
  S_\lambda S_\lambda^* f(\xi) = \int_{\xR^p} K(\xi, \xi') f( \xi')\, d\xi'
  \end{equation}
  with
  $$ K(\xi, \xi') = \int_{\xR^k} e^{i\lambda(\psi(x, \xi) - \psi(x, \xi'))}b(x, \xi) \overline{b}(x, \xi')\, dx.$$
  Let us set
  $$c(x, \xi, \xi') = N(x,\xi)^{-1} \frac{\xi-\xi'} {\vert \xi-\xi' \vert}.$$
  Then we can write
  \begin{equation}\label{L1}
    c(x, \xi,\xi')\cdot \nabla_x e^{i\lambda(\psi(x, \xi) - \psi(x, \xi'))} = e^{i\lambda(\psi(x, \xi) - \psi(x, \xi'))}i \lambda \Delta(x, \xi,\xi')
    \end{equation}
  where
  \begin{align*}
  \Delta(x, \xi,\xi')&=  \sum_{j=1}^k c_j(x, \xi,\xi')\big(\frac{\partial \psi}{\partial x_j}(x, \xi) -\frac{\partial \psi}{\partial x_j}(x, \xi')\big),\\ 
   & = \sum_{j,l=1}^k c_j(x, \xi,\xi')\big( \frac{ \partial^2 \psi}{\partial x_j\partial \xi_l}(x, \xi)(\xi_l - \xi'_l) + \mathcal{O}(\vert \xi  - \xi' \vert^2 \big), \\
    & = \langle N(x, \xi) c(x, \xi, \xi'), \xi - \xi' \rangle + \mathcal{O}(\vert \xi  - \xi' \vert^2) 
     = \vert \xi- \xi' \vert +   \mathcal{O}(\vert \xi  - \xi' \vert^2),
   \end{align*}
  where $\mathcal{O}(\vert \xi  - \xi' \vert^2)$ is independent of $(y, \eta)$. Since    $b$ has small support in   $ \xi$ we deduce that
  \begin{equation}\label{Delta}
   \Delta(x, \xi,\xi') \geq C  \vert \xi- \xi' \vert.
   \end{equation}
  Moreover since the derivatives with respect to $x$ of $N(x,\xi)^{-1}$ are products of $N(x,\xi)^{-1}$ and derivatives of  $N(x,\xi),$ we see that all the derivatives with respect to $x$ of $\Delta(x, \xi,\xi')$ are uniformly bounded in $(y,\eta)$ near $(y_0, \eta_0)$.
  Let us set 
  $$L =   \frac{1}{ i \lambda \Delta(x, \xi,\xi')}c(x, \xi,\xi')\cdot \nabla_x. $$
  It follows from \eqref{L} and the fact that $b$ has compact support in $x$ that for every $N \in \xN$ we can write
  $$ K(\xi, \xi') = \int_{\xR^p} e^{i\lambda(\psi(x, \xi) - \psi(x, \xi'))}\big(^t\!  L)^N [b(x, \xi) \overline{b}(x, \xi')]\, dx.$$
  We deduce from \eqref{Delta} that for every $N \in \xN$ there exists $C_N>0$ independent of $(y, \eta)$ such that
  $$\vert K(\xi, \xi') \vert \leq \frac{C_N}{(1 + \lambda \vert \xi - \xi' \vert)^N}.$$
  Taking $N>p$ we deduce from \eqref{SS*} and Schur lemma that \eqref{trick} holds with a constant $C$ independent of $(y,\eta).$ This completes the proof.
  \end{proof}
  \begin{lemm}\label{hessien}
  Let $ d \geq 1,$ $ \delta \in \xR $  and $\varphi_0 (x,x') = \big( \sum_{j=1}^d   (x_j-x'_j)^2 +  \delta^2 \big)^\mez$. Let  $M=  \Big( \frac{\partial^2 \varphi_0}{\partial x_j  \partial x'_k}(x,x')\Big)_{1 \leq j,k \leq d}.$ Then
\begin{align*}
   &(i)  \, \text{ if }  \delta\neq 0  \quad  M  \text{ has rank }  d  \quad \text{for all } x,x' \in \xR^d,\\
    &(ii)  \, \text{ if }   \delta =0  \quad  M  \text{ has rank }   d-1  \text{ for }  x\neq x'. 
\end{align*}
 \end{lemm}
\begin{proof}
$(i)$ \quad 
A simple computation shows that    
$$M= \varphi_0 (x,x')^{-1}(- \delta_{jk} + \omega_j \omega_k ), \quad \omega_j = \frac{x_j-x'_j}{\varphi_0 (x,x')}$$
where $\delta_{jk}$ is the Kronecker symbol. For $\lambda \in \xR$ consider the  polynomial in $\lambda$
$$F(\lambda) = \det \big( -\delta_{jk} + \lambda \omega_j \omega_k    \big)_{1 \leq j,k \leq d} $$
We have obviously $F(0) = (-1)^d.$  Now  denote by $C_j(\lambda)$ the $j^{th}$ column of this determinant. Then
$$
F'(\lambda) = \sum_{k=1}^{d} \det \big(C_1(\lambda), \ldots, C'_k(\lambda),\ldots C_{d }(\lambda) \big).
$$
 Since $ \det \big(C_1(0), \ldots, C'_k(0),\ldots C_d(0) \big) = (-1)^{d-1}\omega_j^2 $ we obtain $F'(0) = (-1)^{d-1}\sum_{j=1}^{d} \omega_j^2.$
  Now   $C_j(\lambda)$ being  linear with respect to $\lambda$ 
we have $C''_j(\lambda) = 0.$ Therefore
$$
F''(\lambda) = \sum_{j=1}^{d} \sum_{k=1,k\neq j}^{d} 
\det \big(C_1(\lambda),\ldots, C'_j(\lambda),\ldots,C'_k(\lambda),\ldots, C_{d}(\lambda) \big).
$$
Since $C'_j(\lambda) = \omega_j(\omega_1,\ldots,\omega_d)$ 
and $C'_k(\lambda) = \omega_k(\omega_1,\ldots,\omega_d)$  
we have  $F''(\lambda) = 0$ for all $\lambda \in \xR.$ It follows that 
  $F(\lambda) = (-1)^{d}(1 - \lambda  \sum_{j=1}^{d} \omega_j^2).$ Therefore 
  $$\det M = (-1)^{d}(1 -   \sum_{j=1}^{d} \omega_j^2) = (-1)^{d} \frac{\delta^2}{\varphi_0(x,x')^2} \neq 0.$$
$(ii)$  Since $x-x'\neq0$ we may assume without loss of generality that $\omega_d \neq 0.$  Set 
$$ A= \big( - \delta_{jk}  + \omega_j \omega_k \big)_{1 \leq j,k \leq d-1}.$$
Introducing $ G(\lambda) = \det \big( -\delta_{jk} + \lambda \omega_j \omega_k    \big)_{1 \leq j,k \leq d-1} $the same computation as above shows  that   
   $$\det \,A  = (-1)^{d-1}(1 - \sum_{j=1}^{d-1} \omega_j^2) = (-1)^{d-1}\omega_d^2 \neq 0.$$
 \end{proof}

\def\cprime{$'$} \def\cprime{$'$}

%\bibliographystyle{plain}   
%\bibliography{/Users/nicolas/Dropbox/Biblio/Biblio}

 %\bibliography{/Users/claude/Biblio/Biblio}
\end{document}